\documentclass[oneside,english]{amsart}
\usepackage[T1]{fontenc}
\usepackage[latin9]{inputenc}
\usepackage{amsthm}
\usepackage{amstext}
\usepackage{amssymb}
\usepackage{esint}

\makeatletter
\numberwithin{equation}{section}
\numberwithin{figure}{section}
\theoremstyle{plain}
\newtheorem{thm}{Theorem}[section]
  \theoremstyle{definition}
  \newtheorem{defn}[thm]{Definition}
  \theoremstyle{plain}
  \newtheorem{lem}[thm]{Lemma}
  \theoremstyle{plain}
  \newtheorem{prop}[thm]{Proposition}
  \theoremstyle{remark}
  \newtheorem{rem}[thm]{Remark}
 \theoremstyle{definition}
  \newtheorem{example}[thm]{Example}
  \theoremstyle{plain}
  \newtheorem{cor}[thm]{Corollary}

\makeatother

\usepackage{babel}

\begin{document}

\title{Super currents and tropical geometry}
\begin{abstract}
We introduce the formalism of positive super currents on $\mathbb{R}^{n}$,
in strong analogy with the theory of positive currents in $\mathbb{C}^{n}$.
We consider intersection of currents and Lelong numbers, and as an
application we show that the formalism can be used to describe tropical
varieties. This is similar in spirit to the fact that in complex analysis
the current of integration of an analytic variety can be identified
with a closed, positive current.
\end{abstract}

\author{Aron Lagerberg}

\email{aronl@chalmers.se}

\curraddr{Chalmers University of Technology and University of Gothenburg, Department
of Mathematics, 412 96 Gothenburg, Sweden }

\maketitle
\tableofcontents{}

In complex analysis, the natural counterpart of convexity in the real
setting, is that of plurisubharmonicity, and there are many simliarities
between convex- and plurisubharmonic functions. For instance, a smooth
function defined on $\mathbb{R}^{n}$ is convex if and only if its
Hessian is positive definite, and a smooth function defined on $\mathbb{C}^{n}$
is plurisubharmonic iff its (complex) Hessian is positive definite.
On the complex side, a natural way of studying plurisubharmonicity
is provided by the framework of so called positive currents. In fact,
a closed $(1,1)-$current is positive iff it locally can be represented
as $i\partial\bar{\partial}\varphi$ for a plurisubharmonic function
$\varphi$. However, one could argue that from the point of view of
geometry, the study of positive currents rather than that of plurisubharmonic
functions, is in a sense more natural. For instance, a variety of
higher co-dimension than 1, relates to closed, positive currents of
higher bi-degree. The aim of this paper is to introduce a notion of
positive currents correspondning to convex functions defined on $\mathbb{R}^{n}$.
This is carried out by letting ourselves be inspired by the complex
setting (indeed, many of the results and ideas will probably be familiar
to the mathematician knowledgeable in pluripotential theory). The
ideas can be seen as a continuation of those developed in \cite{Berndtsson}.
We then consider the framework of positive currents in the context
of tropical geometry, proving, amongst other things, that every tropical
hypersurface corresponds to a positive current satisfying certain
criterions. This is similar in spirit to the fact that in complex
analysis, a complex hypersurface can be represented by a positive
current satisfying certain hypotheses. Our hope is for this work to
provide a useful tool for attacking problems within tropical geometry,
and for it to serve as a gateway between complex analysis and tropical
geometry.

\subsection*{Acknowledgements. }

I would like to express my gratitude towards my advisor, Bo Berndtsson,
for his invaluable help with writing this article.

\section{Positive super forms and currents in $\mathbb{R}^{n}$}

Let $\mathbb{V}$ and $\mathbb{W}$ denote real vector spaces of dimension
$n$, with coordinates $x=(x_{1},...,x_{n})$ and $\xi=(\xi_{1},...,\xi_{n})$
respectively, for which we fix an isomorphism $J$:$\mathbb{V}\rightarrow\mathbb{W}$,
such that $J(x)=\xi$. We denote its inverse by $J$ as well, so that
$J(\xi)=x$, if $x\in\mathbb{V}$ is the element for which $J(x)=\xi$.
One should think of $\mathbb{V}$ and $\mathbb{W}$ as two copies
of $\mathbb{R}^{n}$ identified via the isomorphism $J$. Let $\mathbb{E}=\mathbb{V}\times\mathbb{W}=\{(x,\xi),x\in\mathbb{V},\xi\in\mathbb{W}\}.$
The map $J$ extends to $\mathbb{E}$ by letting $J(x,\xi)=(J(\xi),J(x)),$
so that $J^{2}=id.$ We consider the space $\mathcal{E}$ of smooth
differential forms on $\mathbb{E}$ whose coefficients only depend
on $x$. Thus, for $x\mapsto\alpha_{KL}(x)$ smooth functions, \begin{equation}
\alpha=\sum_{K,L}\alpha_{K,L}(x)dx_{K}\wedge d\xi_{L}\label{eq:firsteq}\end{equation}
 is such a form, where we use the notation $dx_{K}=dx_{k_{1}}\wedge...\wedge dx_{k_{p}}$
if $K=(k_{1},...,k_{p})$, $|K|$ denotes the length of the vector
$K,$ and similarly for $|L|=q$. We use the convention that we only
sum over indices $K$ and $L$ such that if $K=(k_{1},...,k_{p})$
then $k_{1}<...<k_{p}$, and similarly for the index $L$. If $\alpha$
is of the form \eqref{eq:firsteq}, we say that $\alpha$ is a form
of bi-degree $(p,q)$, and write $\alpha\in\mathcal{E}^{p,q}$ where
$0\leq p,q\leq n$. A form $\alpha$ of degree $(p,p)$ is called
symmetric if $\alpha_{KL}=\alpha_{LK}$ for all indices $K,L$. We
identify the isomorphism $J$ with $J^{*},$which is to say $J(dx_{i})=d\xi_{i},$
and we extend $J$ to an arbitrary $(p,q)-$form by the rule\[
J(\sum_{K,L}\alpha_{K,L}(x)dx_{K}\wedge d\xi_{L})=\sum_{K,L}\alpha_{K,L}(x)d\xi_{K}\wedge dx_{L}.\]
Thus a $(p,p)-$form $\alpha$ is symmetric iff $J(\alpha)=(-1)^{p}\alpha$.
Note that we use the letter $J$ to denote several, slightly different
maps, but we hope that no confusion will arise. Finally we put $\omega=\sum_{i=1}^{n}dx_{i}\wedge d\xi_{i}$
and $\omega_{n}=\frac{{1}}{n!}\omega^{n}$. In this article, we will
consider three different notions of positivity for forms:
\begin{defn}
A $(n,n)$-form $v$ is positive if $v=g\omega_{n}$ for some function
$g\geq0$. Let $v$ be a symmetric $(p,p)-$form.

\emph{i)} The $(p,p)-$form $v$ is \emph{weakly positive} if \[
v\wedge\alpha_{1}\wedge J(\alpha_{1})\wedge...\wedge\alpha_{n-p}\wedge J(\alpha_{n-p})\]
 is a positive $(n,n)$-form for every choice of $(1,0)$-forms $\alpha_{1},..,\alpha_{n-p}.$ 

\emph{ii) }We say that the $(p,p)-$form $v$ is \emph{positive},
if \[
v\wedge(\sigma_{n-p})\alpha\wedge J(\alpha)\geq0\]
 for every $(n-p,0)-$form $\alpha$, where \[
\sigma_{k}=(-1)^{\frac{k(k-1)}{2}}.\]

\emph{iii)} Finally, the $(p,p)-$form $v$ is \emph{strongly positive}
if \[
v=\sum_{s}c_{s}\alpha_{1,s}\wedge J(\alpha_{1,s})\wedge...\wedge\alpha_{p,s}\wedge J(\alpha_{p,s})\]
 where $c_{s}\geq0,$ and $\alpha_{j,s}$are $(1,0)$-forms. 
\end{defn}
In the following lemma we collect some elementary properties concerning
positive forms. 
\begin{lem}
\label{lem:elementaryproperties}The following properties hold:

1)\label{lemma_count} For $(1,0)$-forms $\alpha_{i}$,

\[
\alpha_{1}\wedge J(\alpha_{1})\wedge...\wedge\alpha_{p}\wedge J(\alpha_{p})=\]
\[
=\sigma_{p}\alpha_{1}\wedge\alpha_{2}\wedge...\wedge\alpha_{p}\wedge J(\alpha_{1})\wedge...\wedge J(\alpha_{p}).\]

2) If $w_{i}$ are strongly positive forms for $i=1,...,s$ and v
is a positive form, then 

$\,\,\,\,\,\,\,\,\, v\wedge w_{1}\wedge...\wedge w_{s}$ is positive.

3) The wedge product of finitely many symmetric forms is symmetric. 

4) We have the following inclusions: \[
\{\text{strongly positive forms}\}\subset\{\text{positive forms}\}\subset\{\text{weakly positive forms}\}.\]
\end{lem}
\begin{proof}
The proof of properties 1-3) are elementary, and left to the reader.
Let us prove property $4)$. Let $v$ be a positive $(p,p)-$form.
If $\alpha_{i}$ are $(1,0)-$forms for $1\leq i\leq n-p$, property
1) above gives us the equality \[
0\leq v\wedge\alpha_{1}\wedge J(\alpha_{1})\wedge...\wedge\alpha_{n-p}\wedge J(\alpha_{n-p})=v\wedge(\sigma_{n-p})\cdot\alpha_{1}\wedge...\wedge\alpha_{n-p}\wedge J(\alpha_{1}\wedge...\wedge\alpha_{n-p}),\]
and thus $v$ is weakly positive, which proves the second inclusion.
For the first inclusion, we note that for every $(p,0)-$form $\beta,$\[
\alpha_{1}\wedge J(\alpha_{1})\wedge...\wedge\alpha_{p}\wedge J(\alpha_{p})\wedge\sigma_{n-p}\beta\wedge J(\beta)=\sigma_{n-p}\sigma_{p}\alpha\wedge J(\alpha)\wedge\beta\wedge J(\beta)=\]
\[
=\sigma_{n-p}\sigma_{p}(-1)^{p(n-p)}\alpha\wedge\beta\wedge J(\alpha\wedge\beta),\]
where $\alpha=\alpha_{1}\wedge...\wedge\alpha_{n-p}$. A simple calculation
shows that $\sigma_{n-p}\sigma_{p}(-1)^{p(n-p)}=\sigma_{n}$, and
since $\alpha\wedge\beta\wedge J(\alpha\wedge\beta)=\sigma_{n}c^{2}\omega_{n}$
for some constant $c$, we finally obtain that \[
\alpha_{1}\wedge J(\alpha_{1})\wedge...\wedge\alpha_{p}\wedge J(\alpha_{p})\wedge\sigma_{n-p}\beta\wedge J(\beta)\geq0,\]
which proves the first inclusion. 
\end{proof}
The property of a form being positive is reflected in the associated
matrix of the form:
\begin{prop}
\label{lem:postivedefinite}Let $\alpha=\sum_{K,L}\alpha_{KL}(\sigma_{p}\cdot dx_{K}\wedge d\xi_{L})$
be a symmetric $(p,p)$-form. Then $\alpha$ is positive iff the matrix
$(\alpha_{KL})_{K,L}$ is positive definite. Moreover, if $\alpha$
is positive, we can find $(p,p)$-forms $\gamma_{k}$ for which \[
\alpha=\sum_{K}\alpha_{KK}(\sigma_{p}\cdot\gamma_{K}\wedge J(\gamma_{K})),\]
where $\alpha_{KK}\geq0$, for each $K.$ \end{prop}
\begin{proof}
Let $\beta=\sum_{|K|=p}\beta_{K}dx_{K^{c}}$ be a $(n-p,0)-$form,
where $K^{c}$ denotes the complementary multi-index to $K$. Then
\[
\alpha\wedge(\sigma_{n-p}\cdot\beta\wedge J(\beta))=\sum_{|K|,|L|=p}\sigma_{n-p}\sigma_{p}\cdot\alpha_{KL}\beta_{K}\beta_{L}dx_{K}\wedge d\xi_{K}\wedge dx_{K^{c}}\wedge d\xi_{K^{c}}.\]
Thus, since $dx_{K}\wedge d\xi_{K}\wedge dx_{K^{c}}\wedge d\xi_{K^{c}}=(-1)^{n(n-p)}\sigma_{n}\omega_{n}$
and since \[
\sigma_{n-p}\sigma_{p}(-1)^{n(n-p)}\sigma_{n}=1\]
 by direct calculations, we obtain \[
\alpha\wedge(\sigma_{n-p}\cdot\beta\wedge J(\beta))=\sum_{|K|,|L|=p}\sigma_{n-p}\sigma_{p}(-1)^{n(n-p)}\sigma_{n}\alpha_{KL}\beta_{K}\beta_{L}\omega_{n}=\]
\[
=\sum_{|K|,|L|=p}\alpha_{KL}\beta_{K}\beta_{L}\omega_{n}.\]
Hence, if $(\beta)=(\beta_{K})_{|K|=p}$ is the vector corresponding
to the form $\beta$, \[
\alpha\wedge(\sigma_{n-p}\cdot\beta\wedge J(\beta))=(\beta)^{t}(\alpha_{IJ})(\beta)\omega_{n}.\]
It follows that $\alpha$ is positive if and only if $(\alpha_{KL})_{K,L}$
is positive definite. The second statement follows immediately from
the spectral theorem for positive definite, symmetric matrices. 
\end{proof}
The set of weakly positive and strongly positive forms are convex
cones, by definition dual under the paring

\[
(v,w)\mapsto v\wedge w,\]
where $v$ is a weakly positive $(p,p)$-form and $w$ is a strongly
positive $(q,q)$-form, and $p+q=n:$ By definition, $v$ is weakly
positive iff $v\wedge w$ is positive for each strongly positive $w$,
and since the bidual of a convex cone is equal to the closure of the
cone, we see that $w$ is strongly positive iff $w\wedge v$ is positive
for every weakly positive $v.$ Moreover, one can show that the cone
of positive forms is self-dual. At this point, we introduce the useful
notation $\widehat{dx_{i}}=dx_{1}\wedge....\wedge dx_{i-1}\wedge dx_{i+1}\wedge...dx_{n}$.
\begin{lem}
A symmetric, weakly positive (n-1,n-1)-form is strongly positive.
The same applies for symmetric $(1,1)$-forms.\end{lem}
\begin{proof}
Let $\alpha=\sum_{i,j=1}^{n}\alpha_{ij}\sigma_{n-1}\widehat{dx_{i}}\wedge\widehat{d\xi_{j}}$
be a symmetric, weakly positive $(n-1,n-1)$-form. By definition such
a form is also positive. Thus, by Proposition \ref{lem:postivedefinite}
we can assume that \[
\alpha=\sum_{i=1}^{n}\alpha_{ii}\sigma_{n-1}\widehat{dx_{i}}\wedge\widehat{d\xi_{i}},\]
with $\alpha_{ii}\geq0$. Thus, using property 1) of Lemma \ref{lem:elementaryproperties},
we see that \[
\alpha=\sum_{i=1}^{n}\alpha_{ii}(\sigma_{n-1})^{2}dx_{1}\wedge d\xi_{1}\wedge....\wedge dx_{i-1}\wedge d\xi_{i-1}\wedge dx_{i+1}\wedge d\xi_{i+1}\wedge...\wedge dx_{n}\wedge\xi_{n}\]
and hence, $\alpha$ is strongly positive which proves the first statement.
The second statement is a consequence of the duality between the convex
cones of weakly positive respectively strongly positive forms: Let
$\alpha$ be a weakly positive $(1,1)-$form. Then, since we have
just proved that every weakly positive $(n-1,n-1)-$forms is strongly
positive, we see that $\alpha\wedge\beta\geq0$ for every weakly positive
$(n-1,n-1)-$form $\beta$. By duality, this implies that $\alpha$
is strongly positive.
\end{proof}
In particular, the Lemma implies that the three different notions
of positivity coincides for forms of bi-degree $(0,0),(1,1),(n-1,n-1)$
and $(n,n)$. Thus, for such forms we will usually only use the epithet
{}``positive''. 

Let $\mathbb{V}',\mathbb{W}'$ be a real vector spaces of dimension
$m$, between which we fix an isomorphism $J'$ as above, and let
$\psi:\mathbb{V}\rightarrow\mathbb{V}'$ be an affine map. We let
$\mathbb{E}'=\mathbb{V}'\times\mathbb{W}'$. Then $\psi$ extends,
uniquely, to an affine map from $\mathbb{E}$ to $\mathbb{E}'$, which
we denote by $\tilde{\psi},$ by demanding that $\tilde{\psi}\circ J=J'\circ\tilde{\psi}$.
We can now define the pull-back operator $\tilde{\psi}^{*}:\mathcal{E}(\mathbb{E}')\rightarrow\mathcal{E}(\mathbb{E})$,
by letting \[
\tilde{\psi}^{*}(\sum_{I,J}\alpha_{IJ}dx_{I}\wedge d\xi_{J})=\sum_{I,J}(\alpha_{IJ}\circ\psi)\psi^{*}(dx_{I})\wedge\psi^{*}(d\xi_{J})\]
where, if $I=(i_{1},..,i_{p}),$ we as usual let $\psi^{*}(dx_{I})=\psi^{*}(dx_{i_{1}})\wedge...\wedge\psi^{*}(dx_{i_{p}}),$
and analogously for $\psi^{*}(d\xi_{J})$. Observe that the pull-back
operator commutes with the operator $J$, and also with the operator
$d$. When no confusion seems likely to arise, we will denote the
extension $\tilde{\psi}$ by $\psi$ as well. Note that if we did
not demand $\psi$ to be affine, the pullback of a form in $\mathcal{E}$
could have coefficients depending on the variable $\xi$, and thus
not be a map from $\mathcal{E}$ to $\mathcal{E}$. If $\psi:\mathbb{E}\rightarrow\mathbb{E}$
is an affine map, an easy computation shows that \begin{equation}
\tilde{\psi}^{*}(\omega_{n})=|det(\psi)|^{2}\omega_{n}.\label{eq:determinantpullback}\end{equation}
This implies that if $\psi$ corresponds to a change of coordinates,
$\tilde{\psi}^{*}\omega_{n}=c\omega_{n}$ for some constant $c>0$.
Thus, positivity does not depend on the form $\omega$ which we use
as a reference. If $\psi$ corresponds to an inclusion of a subspace
$\mathbb{V}\subset\mathbb{V}^{'}$, we call $\tilde{\psi}^{*}(\alpha)$
the restriction of the form $\alpha$ to the subspace $\mathbb{V}$. 
\begin{prop}
\label{pro:positiverestriction}With the above notation the following
holds: $\alpha$ is a weakly positive $(p,p)$-form on $\mathbb{E}^{'}$,
iff the restriction of $\alpha$ to every $p$-dimensional subspace
is positive, that is, if $\tilde{\psi}^{*}(\alpha)$ is positive for
every inclusion map $\psi:\mathbb{V}\rightarrow\mathbb{V}^{'}$, where
$\mathbb{V}$ is a $p$-dimensional subspace of $\mathbb{V}^{'}$. \end{prop}
\begin{proof}
Suppose that $\alpha$ is a $(p,p)$-form such that $\psi^{*}\alpha$
is positive for every inclusion map $\psi:\mathbb{V}\rightarrow\mathbb{V}^{'}$
where $\mathbb{V}$ is a $p$-dimensional subspace of $\mathbb{V}^{'}$.
By choosing a basis on $\mathbb{V}'$ we can identify $\mathbb{V}'$
with $\mathbb{R}^{m}$, and regard $\mathbb{V}=\psi(\mathbb{V}')$
as a $p-$dimensional subspace of $\mathbb{R}^{m}$. We need to show
that, for any $(1,0)$-forms $v_{i}$, the number $c$ defined by
$c\omega_{m}:=\alpha\wedge v_{p+1}\wedge J(v_{p+1})\wedge...\wedge v_{m}\wedge J(v_{m})$
satisfies $c\geq0$. To this end, assume that $v_{p+1},...,v_{m}$
are linearly independent $(1,0)$-forms on $\mathbb{E}'$. These correspond
to independent vectors ${e_{p+1},...,e_{m}}$ in $\mathbb{R}^{m}$,
which we can extend to a basis $\{e_{1},...,e_{m}\}$ of $\mathbb{V}'$,
with corresponding forms $v_{i}$. We define two maps: the first,
$\psi:\mathbb{R}^{p}\rightarrow\mathbb{R}^{m}$, is given by $\psi(x_{1},...,x_{p})=\sum_{i=1}^{p}x_{i}e_{i}$
and the second, $\sigma:\mathbb{R}^{m-p}\rightarrow\mathbb{R}^{m}$,
is defined by $\sigma(x_{p+1},...,x_{m})=\sum_{i=p+1}^{m}x_{i}e_{i}$.
Using these two maps, we can define the invertible affine map $\Gamma:\mathbb{R}^{m}\rightarrow\mathbb{R}^{m}$
given by $\Gamma(x_{1},...,x_{m})=\psi(x_{1},...,x_{p})+\sigma(x_{p+1},...,x_{m})$.
Then \begin{equation}
\Gamma^{*}(c\omega_{m})=\Gamma^{*}(\alpha)\wedge\Gamma^{*}(v_{p+1}\wedge J(v_{p+1})\wedge...\wedge v_{m}\wedge J(v_{m}))=\label{eq:gammaeq}\end{equation}
\[
=\psi^{*}(\alpha)\wedge dx_{p+1}\wedge d\xi_{p+1}\wedge...\wedge dx_{m}\wedge d\xi_{m},\]
which, by assumption, is a positive $(m,m)-$form. By formula \eqref{eq:determinantpullback}
we see that $c\geq0$, as desired. Conversely, if $\alpha$ is a weakly
positive $(p,p)$-form on $\mathbb{E}^{'}$ and $\psi:\mathbb{V}\rightarrow\mathbb{V}'$
is a linear map of rank $p$ (which thus corresponds to an inclusion)
we let $\{e_{1},...,e_{m}\}$ be a basis of $\mathbb{V}'\simeq\mathbb{R}^{m}$
such that $\{e_{1},...,e_{p}\}$ spans the column space of $\psi,$
and define $\Gamma$ as above. Then equation \eqref{eq:gammaeq} is
still valid, showing that $\psi^{*}(\alpha)$ is a positive multiple
of $\omega_{p}$. 
\end{proof}
We want to define the integral of an $(n,n)-$form over the space
$\mathbb{E}$. For this, we assume that the vector space $\mathbb{V}$
is endowed with an inner product $(\cdot,\cdot)$, and choose an orthonormal
basis $\{e_{1},...,e_{n}\}$, with corresponding coordinates $(x_{1},...,x_{n})$.
We endow $\mathbb{W}$ with the same structure via the isomorphism
$J$. Then $dx_{1}\wedge....\wedge dx_{n}$ is a $(n,0)$ form on
$\mathbb{V}$, $d\xi_{1}\wedge....\wedge d\xi_{n}$ is an $(0,n)-$form
on $\mathbb{W}$, and every $(n,n)$-form $\alpha$ can be written
as \[
\alpha=\alpha_{0}(x)dx_{1}\wedge....\wedge dx_{n}\wedge d\xi_{1}\wedge...\wedge\xi_{n},\]
for some function $\alpha_{0}$ on $\mathbb{V}$.  
\begin{defn}
The integral of an $(n,n)-$form $\alpha$ as above is given by\[
\int_{\mathbb{E}}\alpha=\int_{\mathbb{V}}\alpha_{0}(x)dx_{1}\wedge....\wedge dx_{n}.\]

\end{defn}
The above definition depends only on the inner product chosen. Indeed,
if we choose a different orthonormal basis, say $\{e_{1}^{'},...,e_{n}^{'}\}$,
the map $\psi$ which sends $e_{i}$ to $e_{i}^{'}$ have determinant
$1$ or $-1$. In either case, by formula \eqref{eq:determinantpullback}
this means that \[
\int_{\mathbb{E}}\tilde{\psi}^{*}(\alpha)=(\pm1)^{2}\int_{\mathbb{E}}\alpha=\int_{\mathbb{E}}\alpha.\]
Thus the definition is independent of which orthonormal basis we choose
to work with. In particular, it does not depend on any orientation
of $\mathbb{V}$. One can also understand the definition as follows:
the choice of inner product allows us, as above, to choose a volume-element
$d\xi_{1}\wedge...\wedge d\xi_{n}$ on $\mathbb{W}$, the total volume
of which we define to be \[
\int_{\mathbb{W}}d\xi_{1}\wedge...\wedge d\xi_{n}=1.\]
Then, by formally applying Fubini's theorem, \[
\int_{\mathbb{E}}\alpha_{0}(x)dx_{1}\wedge....\wedge dx_{n}\wedge d\xi_{1}\wedge...\wedge\xi_{n}=\int_{\mathbb{V}}\alpha_{0}(x)dx_{1}\wedge....\wedge dx_{n}\cdot\int_{\mathbb{W}}d\xi_{1}\wedge...\wedge\xi_{n}=\]
\[
=\int_{\mathbb{V}}\alpha_{0}(x)dx_{1}\wedge....\wedge dx_{n}.\]
If $\Omega\subset\mathbb{V}$ is open, we define \[
\int_{\Omega\times\mathbb{W}}\alpha=\int_{\mathbb{E}}\chi_{\Omega}(x)\cdot\alpha_{0}(x)dx_{1}\wedge....\wedge dx_{n}\wedge d\xi_{1}\wedge...\wedge\xi_{n},\]
 where $\chi_{\Omega}$ denotes the characteristic function of the
set $\Omega.$ By formula \eqref{eq:determinantpullback} we have
the following change of variable formula for a non-singular, affine
map $\psi:\mathbb{V}\rightarrow\mathbb{V}^{'}$:\begin{equation}
\int_{\mathbb{E}}\tilde{\psi}^{*}\alpha=\int_{\mathbb{V}}|det(\psi)|^{2}\alpha_{0}(\psi(x))dx_{1}\wedge....\wedge dx_{n}=|det(\psi)|\int_{\mathbb{E}}\alpha.\label{eq:variablechangeformula}\end{equation}

If $L\subset\mathbb{V}$ is an oriented submanifold of dimension $k$,
and if $\alpha=\sum_{|I|=k}\alpha_{I}(x)dx_{I}\wedge d\xi_{1}\wedge...\wedge d\xi_{n}$
is an arbitrary $(k,n)-$form, we define the integral of $\alpha$
over $L\times\mathbb{W}$ by \[
\int_{L\times\mathbb{W}}(\sum_{|I|=k}\alpha_{I}(x)dx_{I}\wedge d\xi_{1}\wedge...\wedge d\xi_{n})=\sum_{|I|=k}\int_{L}\alpha_{I}(x)dx_{I}.\]

\begin{rem}
It might be interesting at this point to compare with the complex
setting. First of all, the map $J$ could be compared with the usual
complex structure which identifies $\mathbb{C}^{n}$ with $\mathbb{R}^{n}+i\mathbb{R}^{n}$.
Under this identification we compactify the imaginary directions by
considering $\mathbb{R}^{n}+i\mathbb{R}^{n}/\mathbb{Z}^{n}.$ A convex
function $f$ on $\mathbb{R}^{n}$ can then be regarded as a plurisubharmonic
function on $\mathbb{R}^{n}+i\mathbb{R}^{n}/\mathbb{Z}^{n},$ by demanding
the extension to be independent of the imaginary directions. In a
similiar way, we can consider a real $n-$form $\alpha$ as a complex
$(n,n)-$form $\tilde{\alpha}$, and in a natural way, \[
\int_{\mathbb{R}^{n}+i\mathbb{R}^{n}/\mathbb{Z}^{n}}\mbox{\ensuremath{\tilde{\alpha}}}=\int_{\mathbb{R}^{n}}\alpha,\]
 which could be seen as an analogue of our definition, \[
\int_{\mathbb{R}^{n}}d\xi_{1}\wedge....\wedge d\xi_{n}=1.\]
Generalizing this slightly, we use another lattice in order to compactify
the imaginary directions. Thus, if $\Gamma$ is any lattice, we consider
$\mathbb{R}^{n}+i\mathbb{R}^{n}/\Gamma$. Such a lattice induces an
inner product on $\mathbb{R}^{n}$. Indeed, $\Gamma$ is isomorphic
to $\mathbb{Z}^{n}$ via an affine map, and there is a one-to-one
correspondence between affine maps and inner products. In this sense,
we can say that choosing an inner product on $\mathbb{R}^{n}$ (thereby
defining integration of $(p,p)-$forms along $p$-dimensional subspaces
of $\mathbb{E}$) corresponds to compactifying using a lattice as
above, and vice versa. 
\end{rem}
We define the operator $d:\mathcal{E}^{p,q}\rightarrow\mathcal{E}^{p+1,q}$
by the formula \[
d(\sum_{|I|=p,|J|=q}\alpha_{IJ}dx_{I}\wedge d\xi_{J})=\sum_{|I|=p,|J|=q}(\sum_{i=1}^{n}\frac{\partial\alpha_{IJ}(x)}{\partial x_{i}}dx_{i}\wedge dx_{I}\wedge d\xi_{J}).\]
In our setting, we have the following version of Stokes' formula.
\begin{prop}
\label{pro:stokes}For $\Omega\subset\mathbb{V}$ a smoothly bounded,
open subset, and $\alpha$ an $(n-1,n)$-form on $\mathbb{E}$, \[
\int_{\Omega\times\mathbb{W}}d\alpha=\int_{\partial\Omega\times\mathbb{W}}\alpha.\]
\end{prop}
\begin{proof}
By the usual Stokes' formula, we have that \[
\int_{\Omega\times\mathbb{W}}d\alpha=\sum\int_{\Omega\times\mathbb{W}}d(\alpha_{IJ}(x)dx_{I})\wedge d\xi_{J}=\sum\int_{\Omega}d(\alpha_{IJ}(x)dx_{I})\int_{\mathbb{W}}d\xi_{J}=\]

\[
=\sum\int_{\partial\Omega}(\alpha_{IJ}(x)dx_{I})\int_{\mathbb{W}}d\xi_{J}=\int_{\partial\Omega\times\mathbb{W}}\alpha.\]

\end{proof}
We define the operator $d^{\#}:\mathcal{E}^{p,q}\rightarrow\mathcal{E}^{p,q+1}$
by \[
d^{\#}=J\circ d\circ J,\]
which in coordinates is equivalent to \[
d^{\#}=\sum_{j=1}^{n}\partial_{x_{j}}\wedge d\xi_{j}.\]
As immediately follows from the definition, $d^{2}=(d^{\#})^{2}=0,$
and moreover, $d^{\#}\circ J=J\circ d$.

\subsection{Currents and positivity}

In this section, we assume the reader to be familiar with the basic
theory of currents, but we include some proofs to illustrate the setting
in which we work. The precise definition of a current is tedious and
almost identical to the complex case, so we refer to \cite{Demailly}
for the details. The basic idea is that by introducing a topology
on $\mathcal{D}^{p,q}=\{\alpha\in\mathcal{E}^{p,q};\mbox{\ensuremath{\alpha}}\mbox{\text{ has compact support}}\}$,
we can consider the topological dual of $\mathcal{D}^{p,q}$, which
we define to be the space of currents $\mathcal{D}_{n-p,n-q}$. Suffice
it so say that an element of $\mathcal{D}_{n-p,n-q}$ can be viewed
as a $(n-p,n-q)-$form whose coefficients are distributions which
only acts on $\mathbb{V},$ that is, the coefficients are {}``independent
of $\xi$''. Thus every $T\in\mathcal{D}_{p,q}$ can be written as
\[
T=\sum_{|I|=p,|J|=q}T_{IJ}dx_{I}\wedge d\xi_{J},\]
where $T_{IJ}$ are uniquely defined distributions on $\mathbb{V}$.
We denote the paring between an element $\alpha\in\mathcal{D}^{p,q}$
and $T\in\mathcal{D}_{n-p,n-q}$ by $\left\langle T,\alpha\right\rangle ,$
and we use the convention that \[
\left\langle T_{IJ}dx_{I}\wedge d\xi_{J},\alpha_{0}dx_{I^{c}}\wedge d\xi_{J^{c}}\right\rangle =\pm\left\langle T_{IJ},\alpha_{0}\right\rangle ,\]
where the sign is determined by the sign of the permutation sending
$dx_{I}\wedge d\xi_{J}\wedge dx_{I^{c}}\wedge d\xi_{J^{c}}$ to $\omega_{n}$.
For convenience, when a current acts on an element of $\mathcal{E}$,
we always assume this element to have compact support, without explicitly
stating so. As usual we can define $dT,d^{\#}T,dd^{\#}T$ by $\left\langle dT,\alpha\right\rangle =\pm\left\langle T,d\alpha\right\rangle $
and so forth. Thus we say that a current $T$ is \emph{$d-$closed
}if $dT=0,$ and similarly for $d^{\#}.$ These operators so defined,
act continuously on the space of currents.

Now, let $\rho$ be a smooth, radial function with support in the
unit ball, satisfying $\int\rho=1,$ and put $\rho_{\epsilon}(x)=\frac{1}{\epsilon^{n}}\rho(\frac{x}{\epsilon}),$
for $\epsilon>0.$ If we consider the convolution of a current $T=\sum_{|I|=p,|J|=q}T_{IJ}dx_{I}\wedge d\xi_{J}$
with this function $\rho_{\epsilon}$, defined in the usual way as
\begin{equation}
T\star\rho_{\epsilon}=\sum_{|I|=p,|J|=q}(T_{IJ}\star\rho_{\epsilon})dx_{I}\wedge d\xi_{J},\label{eq:convolutionregularize}\end{equation}
 then ${\{T\star\rho_{\epsilon}\}}_{\epsilon>0}$ defines a family
in $\mathcal{E}^{p,q}$ converging weakly to the current $T,$ as
$\epsilon\rightarrow0.$ We call this family a regularization of the
current $T$, and it is easy to see that if $dT=0,$ then $d(T\star\rho_{\epsilon})=0,$
so regularization preserves the property of being closed. 
\begin{lem}
Let $\Omega\subset\mathbb{V}$ be an open set. For \label{cohomologycurrentlemma}a
$d-$closed current $T\in\mathcal{D}_{p,q}(\Omega\times\mathbb{W})$
there exists a $T'\in\mathcal{E}^{p,q}(\Omega\times\mathbb{W})$ such
that $T-T'=dR$ for some $R\in\mathcal{D}_{p-1,q}(\Omega\times\mathbb{W}).$\end{lem}
\begin{proof}
Let \[
T=\sum_{|I|=p,|J|=q}T_{IJ}dx_{I}\wedge d\xi_{J}.\]
 If we denote by $S_{J}$ the $(p,0)$-currents $\sum_{|I|=p}T_{IJ}dx_{I}$,
then \[
T=\sum_{|J|=q}S_{J}\wedge d\xi_{J},\]
and by the hypothesis $dS_{J}=0.$ It is well known from theory of
currents in $\mathbb{V}$, that for every such $S_{J}$ there is a
smooth $(p,0)-$form $S_{J}^{'}$ such that $S_{J}-S_{J}^{'}=dR_{J}$
for some $(p-1,0)-$current $R_{J}$ (where we identify $(p,0)-$currents
on $\mathbb{E}$ with $p-$currents on $\mathbb{V}$). Thus, if we
let $T^{'}=\sum_{|J|=q}S_{J}^{'}\wedge d\xi_{J},$ and $R=\sum_{|J|=q}R_{J}\wedge d\xi_{J}$
we have that $T^{'}$ is smooth and \[
T-T^{'}=\sum_{|J|=q}(S_{J}-S_{J}^{'})\wedge d\xi_{J}=\sum_{|J|=q}dR_{J}\wedge d\xi_{J}=d\sum_{|J|=q}R_{J}d\xi_{J}=dR,\]
as required.\end{proof}
\begin{prop}
\label{pro:poincarecurrentprop}Let $\Omega\subset\mathbb{V}$ be
a star shaped open subset. If $T\in\mathcal{D}_{p,q}(\Omega\times\mathbb{W})$
is d-closed, in the sense that $dT=0,$ and $q\geq1,$ then there
exists an element $\tilde{T}\in\mathcal{D}_{p-1,q}(\Omega\times\mathbb{W})$
such that $d\tilde{T}=T.$ An analogous statement holds if $T$ is
instead $d^{\#}-$closed.\end{prop}
\begin{proof}
By Lemma \ref{cohomologycurrentlemma}, it suffices to show the proposition
in the case where $T$ has smooth coefficients, that is, $T\in\mathcal{E}^{p,q}.$
To this end assume that \[
T=\sum_{|I|=p,|J|=q}T_{IJ}dx_{I}\wedge d\xi_{J},\]
with $T_{IJ}$ smooth. If we, as above, denote by $S_{J}$ the $(p,0)$-forms
$\sum_{|I|=p}T_{IJ}dx_{I}$, then $T=\sum_{|J|=q}S_{J}\wedge d\xi_{J}$,
and since $dT=\sum_{|J|=q}(dS_{J})\wedge d\xi_{J}=0$ by the hypothesis,
we see that $dS_{J}=0.$ Since $\Omega$ is star shaped, the Poincaré
lemma tells us that there exists $(p-1,0)-$forms, $\tilde{S}_{J}$
such that $d\tilde{S}_{J}=S_{J}$. Thus, if we consider the $(p-1,q)-$form
$\tilde{T}=\sum_{|J|=q}\tilde{S}_{J}\wedge d\xi_{J}$, we see that
it satisfies \[
d\tilde{T}=\sum_{|J|=q}d\tilde{S}_{J}\wedge d\xi_{J}=\sum_{|J|=q}S_{J}\wedge d\xi_{J}=T.\]
If $T$ is instead $d^{\#}-$closed, the same argument as above still
applies with the obvious changes.
\end{proof}
We now come to the corresponding notions of positivity for currents.
\begin{defn}
We say that a symmetric $(p,p)-$current $T$ is \emph{weakly positive}
if \[
\left\langle T,\alpha\right\rangle \geq0\]
for each smooth \emph{strongly positive} $(n-p,n-p)-$form $\alpha$
with compact support. $T$ is positive if \[
\left\langle T,\sigma_{n-p}\beta\wedge J(\beta)\right\rangle ,\]
for every smooth $(p,0)$- form $\beta$ with compact support. 
\end{defn}
We postpone the definition of a strongly positive $(p,p)-$current
until section \ref{sec:Intersection-theory-of}. 
\begin{prop}
A function $f:\mathbb{V}\rightarrow\mathbb{R}$ is convex iff $dd^{\#}f$
is a positive (1,1)-current.\end{prop}
\begin{proof}
This is clear if $f$ is smooth since the matrix associated to $dd^{\#}f$
is the Hessian of $f$, so we can apply Proposition \ref{lem:postivedefinite}.
The general case follows by approximation: if $f$ is convex but not
smooth, we can find a family $\{f_{\epsilon}\}_{\epsilon}$ of smooth,
convex functions such that $f_{\epsilon}\rightarrow f.$ Using the
definition of currents, we see that $dd^{\#}f_{\epsilon}\rightarrow dd^{\#}f$
in the weak sense, and thus, $\left\langle dd^{\#}f,\alpha\right\rangle =\lim_{\epsilon\rightarrow0}\left\langle dd^{\#}f_{\epsilon},\alpha\right\rangle \geq0,$
for every positive $(n-1,n-1)-$form $\alpha.$ Hence $dd^{\#}f\geq0$.
Conversely, if $dd^{\#}f\geq0$, we put $f_{\epsilon}(x)=\int f(y)\rho_{\epsilon}(x-y)$
where $\rho_{\epsilon}$ is the regularizing kernel from above. One
easily verifies that $dd^{\#}f_{\epsilon}\geq0$ and hence $f_{\epsilon}$
is convex. Moreover, $f_{\epsilon}$ is smooth, $f_{\epsilon}\geq f$
and $f_{\epsilon}\rightarrow f$ uniformly on compacts. Thus $f$
is convex, as desired. 
\end{proof}
Of course, if $f$ is also smooth, then $dd^{\#}f$ is a positive,
closed $(1,1)-form.$ The following proposition is fundamental for
what is to come; it is the counterpart in our setting of the so called
$dd^{c}$-lemma from complex analysis:
\begin{prop}
\label{pro:ddsharplemma}($dd^{\#}-lemma)$ Let $T$ be a closed,
positive (1,1)-current on $\mathbb{E}$. Then there exists a convex
function $f:\mathbb{V}\rightarrow\mathbb{R}$ for which \[
T=dd^{\#}f.\]
\end{prop}
\begin{proof}
By regularization, we can assume $T$ to be smooth. Since $T$ is
also closed, by Proposition \ref{pro:poincarecurrentprop} we can
find a smooth $(0,1)-$current \[
S=\sum_{i=1}^{n}S_{i}d\xi_{i}\]
 such that $T=dS.$ Moreover, $T$ being symmetric implies $\frac{\partial}{\partial x_{j}}S_{i}=T_{ij}=T_{ji}=\frac{\partial}{\partial x_{i}}S_{j}$,
and so we see that \[
d^{\#}S=\sum_{i,j=1}^{n}\frac{{\partial}}{\partial x_{j}}S_{i}d\xi_{j}\wedge d\xi_{i}=0.\]
 Another application of Proposition \ref{pro:poincarecurrentprop}
provides us with a function $f$ such that $S=d^{\#}f$. We conclude
that \[
T=dd^{\#}f,\]
and since $T$ is positive, $f$ must be convex.
\end{proof}
Using this proposition we can show the following:
\begin{prop}
\label{pro:suppconvex}If $T$ is a closed, positive (1,1)-current,
then each component of $(SuppT)^{c}$ is convex. \end{prop}
\begin{proof}
By the $dd^{\#}$-lemma we can find a convex function $f$ such that
$T=dd^{\#}f.$ If $A$ is a component of $(SuppT)^{c}$, so that $dd^{\#}f=0$
on $A,$ then $f$ is affine on $A.$ For every pair of points $p,q\in A$
we consider the line segment $I$ connecting the two points. The restriction
of $f$ to $A$ is an affine function and the convexity of $f$ implies
that it must be affine on the whole line segment $I$. Since this
is true for every line segment $I$ connecting two points of $A$
we see that $f$ must be affine on the convex hull of $A$, so that
$A\subset conv(A)\subset SuppT^{c}$. Since $A$ was a component we
must have $A=conv(A)$, that is, $A$ is convex.
\end{proof}

\section{Int\label{sec:Intersection-theory-of}ersection theory of currents}

Let $M_{p}$ denote the space of $(p,p)$- forms on $\mathbb{E}$
whose coefficients are measures, endowed with the following topology:
if $T_{i},T\in M_{p}$ then $T_{i}\rightarrow T$ iff $T_{i}(\alpha)\rightarrow T(\alpha)$
for every $(n-p,n-p)-$form $\alpha$ with compact support and whose
coefficients are continuous functions. Note that $M_{p}\subset\mathcal{D}_{p,p}$
as a set, but the topology of $M_{p}$ is stronger than that induced
by $D_{p,p}$. However, a standard proposition in the setting of currents
with measure coefficients, which carries over to our case, is the
following (cf. \cite{Klimek}):
\begin{prop}
\label{pro:weakconvergenceorder}Let $T_{i},T\in M_{p}.$ Then $T_{i}\rightarrow T$
in $M_{p}$ if and only if \[
T_{i}(\alpha)\rightarrow T(\alpha),\]
 for every compactly supported, smooth $(n-p,n-p)-$form $\alpha$,
and if for every compact subset $L\subset\mathbb{R}^{n}$ we have,
\begin{equation}
\sup_{i}\max_{I,J}|(T_{i})_{IJ}|(L)<+\infty.\label{eq:supmax}\end{equation}

\end{prop}
Here, if $\mu$ is a measure $|\mu|$ denotes the total variation
of $\mu.$ Now, let us consider the map given by \[
\Psi_{p}(f_{1},...,f_{p})=dd^{\#}f_{1}\wedge...\wedge dd^{\#}f_{p},\]
where $f_{i}$ is smooth and convex for each index $i$; let us denote
the set of such functions by $K$. We consider $\Psi_{p}$ as a map
from $K^{p}$ to $M_{p}$. Our aim is to show that this map extends,
in a natural way, to a map defined on $p$-tuples of convex functions
(that need not be smooth). By the inclusion $M_{p}\subset\mathcal{D}_{p,p}$,
this extension can be considered as a $(p,p)-$current, which we will
call the intersection product. The scheme to prove this extension
property is the following: first we prove that $\Psi_{p}$ maps bounded
subsets of $K^{p}$ to bounded subsets of $M_{p}$. By the Banach-Alaoglu
theorem, this implies that for every bounded subset $A\subset K^{p}$,
the family $\{\Psi_{p}(x),x\in A\}$ contains a weakly convergent
subsequence. Thus there exists at least one accumulation point of
$\{\Psi_{p}(x),x\in A\}$ in $M_{p}$, and we then show that, in fact,
there exists only one, unique accumulation point. Before we turn to
the details, we need an important property of positive currents:
\begin{prop}
\label{pro:tracemeasureinequality-1}If \[
T=\sum_{|I|=|J|=p}T_{IJ}(\sigma_{p}\cdot dx_{I}\wedge d\xi_{J})\]
 is a symmetric, positive $(p,p)-$current, then each coefficient
$T_{IJ}$ satisfies \[
|\left\langle T_{IJ},\phi\right\rangle |\leq C\cdot\sum_{|I|=p}\left\langle T_{II},\phi\right\rangle ,\]
 for each smooth, non-negative function with compact support, $\phi$.
In particular, $T_{II}$ is a positive measure, and $T_{IJ}$ is a
signed measure, for each multi-indices $I,J$. \end{prop}
\begin{proof}
The argument is clearest when $T$ is smooth, so let us first assume
this is the case. By Proposition \ref{lem:postivedefinite} the $p^{2}\times p^{2}$
matrix $(T_{IJ})$, is positive definite and symmetric at every point
of $\mathbb{R}^{n}$, and thus defines a metric $g$ on $\mathbb{R}^{p^{2}}$.
Let $(e_{I})_{|I|=p}$ be an orthonormal basis of $\mathbb{R}^{p^{2}}$
such that $g(e_{I},e_{J})=T_{IJ}$. Then the Cauchy-Schwartz inequality
gives us \begin{equation}
T_{IJ}=g(e_{I},e_{J})\leq\sqrt{g(e_{I},e_{I})}\cdot\sqrt{g(e_{J},e_{J})}=\sqrt{T_{II}}\cdot\sqrt{T_{JJ}}\leq\frac{T_{II}+T_{JJ}}{2},\label{eq:TijleqTiiplusTjj}\end{equation}
 where we used the inequality between geometric and arithmetic mean
in the last inequality. By exchanging $e_{I}$ for $-e_{I}$ in \eqref{eq:TijleqTiiplusTjj}
we have established the inequality \begin{equation}
|T_{IJ}(x)|\leq C\cdot\sum_{|I|=p}T_{II}(x)\label{eq:tracemeasureinequality}\end{equation}
for some constant $C>0$, which proves the proposition when $T$ is
smooth. If $T$ is not smooth, we can still define an associated metric
as follows: for each smooth, non-negative function with compact support,
$\phi$, we define \[
g(e_{I,}e_{J})=\left\langle T,\sigma_{n-p}\phi dx_{I}\wedge d\xi_{J}\right\rangle ,\]
and extend by linearity. Then $g$ is a positive definite, symmetric
form on $\mathbb{R}^{p^{2}}$, since, if $v=\sum_{I}v_{I}e_{I}\in\mathbb{R}^{p^{2}}$,
then \[
g(v,v)=\left\langle T,\sigma_{n-p}\phi(\sum_{I}v_{I}dx_{I})\wedge J(\sum_{I}v_{I}dx_{I})\right\rangle =\left\langle T,\sigma_{n-p}\tilde{v}\wedge J(\tilde{v)}\right\rangle \geq0,\]
 where $\tilde{v}=\sqrt{\phi}\sum_{I}v_{I}dx_{I}$. Thus, $g(e_{I},e_{I})=\left\langle T_{II},\phi\right\rangle \geq0$,
which implies that $T_{II}$ is a positive measure. Moreover, $g(e_{I},e_{J})=\left\langle T_{IJ},\phi\right\rangle $,
and by the argument used above, \[
|\left\langle T_{IJ},\phi\right\rangle |\leq C\cdot\sum_{|I|=p}\left\langle T_{II},\phi\right\rangle ,\]
 for some constant $C>0$. The proposition follows.
\end{proof}
A subset $A\subset K^{p}$ is \emph{bounded} if for every compact
subset $L$, and every element $(f_{1},...,f_{p})\in A$, we have
that \[
\max_{i}\sup_{x\in L}|f_{i}(x)|\leq C_{L},\]
 for some constant $C_{L}$ (these norms, indexed by $L$, define
the topology of $K^{p}$).
\begin{prop}
\label{pro:nirenberg}If $A\subset K^{p}$ is bounded, then for each
compact set $L\subset\mathbb{R}^{n}$, there exists a constant $D_{L}$
such that if \[
T=dd^{\#}f_{1}\wedge...\wedge dd^{\#}f_{p}\]
 where $(f_{1},...,f_{p})\in A,$ the coefficients of $T$ satisfy
\[
|T_{IJ}|(L):=\int_{L}|T_{IJ}|\leq D_{L}.\]
\end{prop}
\begin{proof}
By the previous proposition, we need only to prove that $T_{II}\leq\tilde{D_{L}}$,
for every $I$. Assume first that $p=1$. Fix a set compact set $L$
and let $f\in K$. Moreover, let $\chi$ be a smooth function equal
to $1$ on $L$ and 0 outside a small neighbourhood of $L$. Then,
since $\partial_{ii}^{2}f\geq0,$ by partial integration we get, \[
|\int_{L}\partial_{ii}^{2}f|\leq|\int_{\mathbb{E}}\chi(x)dd^{\#}f(x)\wedge\widehat{dx_{i}}\wedge\widehat{d\xi_{i}}|=|\int_{Supp\chi\times\mathbb{R}^{n}}f(x)dd^{\#}\chi(x)\wedge\widehat{dx_{i}}\wedge\widehat{d\xi_{i}}|\leq\]
\[
\leq\sup_{x\in Supp\chi}|f(x)|\cdot C_{\chi},\]
using that $\chi$ has uniformly bounded second-order partial derivatives
on every compact subset. Thus, there exists a constant $\tilde{D}_{L}$
such that \[
|T_{ii}|(L)\leq\tilde{D}_{L}\]
for each $i$, proving the case $p=1.$ Assume now that we have proven
the proposition for $p=k-1.$ We want to show that it holds for $p=k$
as well. To this end, let $S=dd^{\#}f_{2}\wedge...\wedge dd^{\#}f_{k}$,
and fix a multi-index $I$ of length $k$. Then, using the same notation
as in the case $p=1$, \[
|\int_{L}(dd^{\#}f_{1}\wedge S)_{II}dV|\leq|\int_{\mathbb{E}}\chi(x)dd^{\#}f_{1}(x)\wedge S\wedge\widehat{dx_{I}}\wedge\widehat{d\xi_{I}}|=\]
\[
=|\int_{Supp\chi\times\mathbb{R}^{n}}f_{1}(x)dd^{\#}\chi(x)\wedge S\wedge\widehat{dx_{I}}\wedge\widehat{d\xi_{J}}|.\]
 By the induction hypothesis, $S$ has coefficients satisfying $|S_{IJ}|(Supp\chi)\leq D_{Supp\chi}$.
Thus\[
|\int_{L}(dd^{\#}f_{1}\wedge S)_{II}dV|\leq\sup_{x\in Supp\chi}|f_{1}(x)|\cdot C_{\chi},\]
and we are done. \end{proof}
\begin{prop}
\label{pro:limitisunique}Let $f_{1},...,f_{p}$ be convex (but not
necessarily smooth) functions, and let, for each $i$, $\{f_{i}^{k}\}_{k}$
be a sequence of smooth, convex functions converging uniformly to
$f_{i}$ on compact subsets. Then the sequence $\{dd^{\#}f_{1}^{k}\wedge...\wedge dd^{\#}f_{p}^{k}\}_{k}\subset M_{p}$
contains a convergent subsequence. If $\{g_{i}^{k}\}_{k}$ is another
sequence of smooth convex functions converging uniformly on compact
subsets, to $f_{i}$ for each $i$, then, if the limits of $dd^{\#}f_{1}^{k}\wedge...\wedge dd^{\#}f_{p}^{k}$
and $dd^{\#}g_{1}^{k}\wedge...\wedge dd^{\#}g_{p}^{k}$ exist, they
must be equal. \end{prop}
\begin{proof}
Let $A=\{(f_{1}^{k},...,f_{p}^{k}),k\geq1\}$. Obviously, the set
$A$ is bounded. Using Proposition \ref{pro:nirenberg}, we see that
for each compactly supported $(n-p,n-p)-$form $\alpha$ with continuous
coefficients, there exists a constant $D_{\alpha}$, for which \[
(dd^{\#}f_{1}^{k}\wedge...\wedge dd^{\#}f_{p}^{k})(\alpha)\leq D_{\alpha}\cdot\max_{IJ}\sup_{x\in Supp\alpha}|\alpha_{IJ}(x)|.\]
Thus, applying the Banach-Alaoglu theorem, we see that the sequence
$\{dd^{\#}f_{1}^{k}\wedge...\wedge dd^{\#}f_{p}^{k}\}_{k}$ contains
a convergent subsequence in $M_{p}$, as desired. To prove the second
statement, we first assume that $p=1$ and let $\alpha$ be a smooth,
compactly supported $(n-1,n-1)-$form . Then, \[
|\int_{\mathbb{E}}(dd^{\#}f_{1}^{k}-dd^{\#}g_{1}^{k})\wedge\alpha|=|\int_{\mathbb{E}}(f_{1}^{k}-g_{1}^{k})\wedge dd^{\#}\alpha|\leq\sup_{x\in Supp\alpha}|f_{1}^{k}-g_{1}^{k}|\cdot C_{\alpha}\]
which tends to $0$ as $k\rightarrow\infty.$ This proves that the
limit is equal in $\mathcal{D}^{1,1}$. However, by Proposition \ref{pro:nirenberg},
both of the forms $dd^{\#}f_{1}^{k}$ and $dd^{\#}g_{1}^{k}$ satisfy
$\eqref{eq:supmax}$, and so, by Proposition \ref{pro:weakconvergenceorder},
$dd^{\#}f_{1}^{k}$ and $dd^{\#}g_{1}^{k}$ converge to the same limit
in $M_{p}$ as well. Now, assume the statement is proved for $p=m-1,$
and let $S^{k}=dd^{\#}f_{2}^{k}\wedge...\wedge dd^{\#}f_{p}^{k}$.
Then, since $S^{k}$ is closed, \[
\int_{\mathbb{E}}(dd^{\#}f_{1}^{k}-dd^{\#}g_{1}^{k})\wedge\alpha\wedge S^{k}=\int_{\mathbb{E}}(f_{1}^{k}-g_{1}^{k})\wedge dd^{\#}\alpha\wedge S^{k},\]
where $\alpha$ is a test-form of degree $(n-p,n-p)$. Moreover, by
Proposition \ref{pro:nirenberg}, the coefficients of $S^{k}$ satisfy
\[
|S_{IJ}^{k}|(Supp\alpha)\leq C_{\alpha}.\]
Thus there exists a constant $D_{\alpha}$ for which, \[
|\int_{\mathbb{E}}(dd^{\#}f_{1}^{k}-dd^{\#}g_{1}^{k})\wedge\alpha\wedge S^{k}|\leq D_{\alpha}\sup_{x\in Supp\alpha}|f_{1}^{k}-g_{1}^{k}|.\]
 This last expression thus tends to 0 as $k\rightarrow\infty.$ Again,
by Proposition \ref{pro:nirenberg} and Proposition \ref{pro:weakconvergenceorder},
we are done. 
\end{proof}
We can now define the intersection product $dd^{\#}f_{1}\wedge...\wedge dd^{\#}f_{p}$,
for $f_{1},...,f_{p}$ convex functions on $\mathbb{R}^{n}$, by using
the continuity of $\Psi_{p}$: it is well known that for a convex
function $f$ one can find a sequence of smooth, convex functions
$f^{k}$ which is monotone in $k$, and which converge to $f$ pointwise.
By Dini's theorem (for its statement, see the discussion after equation
\eqref{eq:homogenizationinequality}), $f^{k}$ converges uniformly
to $f$ on every compact subset of $\mathbb{R}^{n}$. Applying this
for each function $f_{i}$, by using Proposition \ref{pro:limitisunique}
we can define (after possibly choosing a subsequence), \[
dd^{\#}f_{1}\wedge...\wedge dd^{\#}f_{p}=\lim_{k\rightarrow\infty}dd^{\#}f_{1}^{k}\wedge...\wedge dd^{\#}f_{p}^{k},\]
and the definition does not depend on the way we approximate the functions
$f_{i}$ (or which subsequence we choose). 
\begin{defn}
A \emph{strongly positive} $(p,p)-$current is a current of the form
$dd^{\#}f_{1}\wedge...\wedge dd^{\#}f_{p}$. 
\end{defn}
Note that such currents are automatically closed. We collect some
immediate observations about such currents in a proposition: 
\begin{prop}
\textup{The intersection product $dd^{\#}f_{1}\wedge...\wedge dd^{\#}f_{p}$
is weakly positive, it is symmetric in its arguments, and its coefficients
are measures. Moreover, it satisfies the relation }

\begin{equation}
Supp(dd^{\#}f_{1}\wedge...\wedge dd^{\#}f_{p})\subset Supp(dd^{\#}f_{1})\cap...\cap Supp(dd^{\#}f_{p}).\label{eq:supp_of_intersection_subset_intersection_of_supp}\end{equation}

\end{prop}
We also have the following stability property:
\begin{prop}
\label{pro:intersectionstability2}Let $f,g_{1},...,g_{p}$ be convex
functions, where $p<n$. If $\{f_{\epsilon}\}$ is a family of continuous
functions converging pointwise to $f,$ for which $\sup_{\epsilon}\sup_{x\in K}f_{\epsilon}(x)$
is bounded for every compact set $K\subset\mathbb{V}$, then\[
\lim_{\epsilon\rightarrow0}dd^{\#}f_{\epsilon}\wedge dd^{\#}g_{1}\wedge...\wedge dd^{\#}g_{p}=dd^{\#}f\wedge dd^{\#}g_{1}\wedge...\wedge dd^{\#}g_{p}.\]
\end{prop}
\begin{proof}
Since the current $dd^{\#}g_{1}\wedge...\wedge dd^{\#}g_{p}$ has
measure coefficients when written in coordinates, we see that if $\alpha$
is a compactly supported, smooth $(n-p-1,n-p-1)-$form, the $(n,n)$-current
$dd^{\#}g_{1}\wedge...\wedge dd^{\#}g_{p}\wedge dd^{\#}\alpha$ can
be represented by a positive measure on $\mathbb{R}^{n}$ with compact
support. By the dominated convergence theorem, \[
\lim_{\epsilon\rightarrow0}\left\langle dd^{\#}f_{\epsilon}\wedge dd^{\#}g_{1}\wedge...\wedge dd^{\#}g_{p},\alpha\right\rangle =\lim_{\epsilon\rightarrow0}\int_{\mathbb{E}}f_{\epsilon}\wedge dd^{\#}g_{1}\wedge...\wedge dd^{\#}g_{p}\wedge dd^{\#}\alpha=\]

\[
=\int_{\mathbb{E}}f\wedge dd^{\#}g_{1}\wedge...\wedge dd^{\#}g_{p}\wedge dd^{\#}\alpha=\left\langle dd^{\#}f\wedge dd^{\#}g_{1}\wedge...\wedge dd^{\#}g_{p},\alpha\right\rangle ,\]
which proves the claim.\end{proof}
\begin{prop}
\label{pro:intersectionstability}If $f^{1},...,f^{p}$ are convex
functions, and for each $i\in\{1,...,p\}$ there is a family of convex
functions $\{f_{\epsilon_{i}}^{i}\}_{\epsilon_{i}>0}$ such that \[
\lim_{\epsilon_{i}\rightarrow0}f_{\epsilon_{i}}^{i}(x)=f^{i}(x),\]
for every $x\in\mathbb{V}$, and which satisfiy that $\sup_{\epsilon}\sup_{x\in K}f_{\epsilon_{i}}^{i}(x)$
is bounded for every compact set $K\subset\mathbb{V}$ and for each
$i$. Then \begin{equation}
\lim_{\epsilon_{i_{1}}\rightarrow0}...\lim_{\epsilon_{i_{p}}\rightarrow0}dd^{\#}f_{\epsilon_{i_{1}}}^{1}\wedge...\wedge dd^{\#}f_{\epsilon_{i_{p}}}^{p}=dd^{\#}f^{1}\wedge...\wedge dd^{\#}f_{p},\label{eq:intersectionsatbilityeq}\end{equation}
for any permutation $(i_{1},...,i_{p})$ of the n-tuple $(1,...,p)$. \end{prop}
\begin{proof}
The assumption that the families consist entirely of convex functions
ensures that the expression inside the limit in \eqref{eq:intersectionsatbilityeq}
is strongly  positive. Thus, we can apply Proposition \ref{pro:intersectionstability2}
successively to obtain the desired conclusion.\end{proof}
\begin{defn}
The Monge-Ampère measure of a convex function $f$ is the positive
measure defined by\[
MA(f)=(dd^{\#}f)^{n}:=dd^{\#}f\wedge...\wedge dd^{\#}f,\]
 where the product is taken $n$ times.
\end{defn}
Note that we here identify closed, positive $(n,n)-$currents with
positive measures. If $f$ is smooth, then \[
MA(f)=det(\frac{\partial^{2}f}{\partial x_{i}\partial x_{j}})dx_{1}\wedge d\xi_{1}\wedge...\wedge dx_{n}\wedge d\xi_{n}.\]
A very nice paper concerning real Monge-Ampère measures is \cite{Taylor}.
In fact, our approach in this paper could be considered as a generalization
of the formalism defined there in. 
\begin{prop}
Let $f$ b\label{pro:homogenuousMAvanish}e a convex, $1-$homogeneous
function, that is, $f(\lambda x)=\lambda f(x)$ for every $x\in\mathbb{V}$
and $\lambda\in\mathbb{R}$. Then $MA(f)=0$ at every point $x\neq0$.\end{prop}
\begin{proof}
Assume first that $f$ is smooth and fix a point $x\neq0$. The homogeneity
of $f$ implies that there exists a direction in which $f$ is affine.
More precisely, there exists a linear subspace of dimension 1, such
that the restriction of $f$ to this subspace is piecewise affine,
the two pieces being separated by the origin. By an affine change
of coordinates we can thus assume that $\frac{\partial^{2}}{\partial x_{1}\partial x_{1}}f(x)=0,$
that is, one of the eigenvalues of $D^{2}f(x)$ vanishes. This implies
that $MA(f)(x)=0.$ If $f$ is not assumed to be smooth, we choose
a family of 1-homogeneous smooth convex functions such that $f_{i}\rightarrow f$.
By continuity of the Monge-Ampère operator, $MA(f)=\lim_{i\rightarrow\infty}MA(f_{i})=0$. \end{proof}
\begin{example}
We wish to calculate the Monge-Ampère measure of $dd^{\#}|x|$. First
we calculate \begin{equation}
dd^{\#}|x|=d(\frac{1}{2|x|}d^{\#}(|x|^{2}))=\frac{dd^{\#}|x|^{2}}{2|x|}-\frac{d|x|^{2}\wedge d^{\#}|x|^{2}}{4|x|^{3}}.\label{eq:ddshaprabsolutevalueofx}\end{equation}
But the form $\frac{d|x|^{2}\wedge d^{\#}|x|^{2}}{4|x|^{3}}=0$ on
$|x|=r>0$ and so, by using Stokes' theorem, we obtain \[
\int_{B(0,r)\times\mathbb{R}^{n}}(dd^{\#}|x|)^{n}=\int_{\partial B(0,r)\times\mathbb{R}^{n}}\frac{d^{\#}|x|\wedge(dd^{\#}|x|^{2})^{n-1}}{2^{n}|x|^{n}}=\]
\[
=\frac{{1}}{2^{n}r^{n}}\int_{B(0,r)\times\mathbb{R}^{n}}(\sum_{k}2dx_{k}\wedge d\xi_{k})^{n}=\frac{{n!}}{r^{n}}\int_{B(0,r)\times\mathbb{R}^{n}}dx_{1}\wedge d\xi_{1}\wedge...\wedge dx_{n}\wedge d\xi_{n}=\]
\[
=n!Vol_{n}(B(0,1)).\]
Since the above integral is independent of $r$ (or by using proposition
\ref{pro:homogenuousMAvanish}), we see that the measure $(dd^{\#}|x|)^{n}$
equals the Dirac measure at the origin multiplied with a dimensional
constant.
\end{example}
A useful class of convex functions are those that grow {}``at most
linearly at infinity''. By the similarity with the complex setting,
we define the\emph{ Lelong class} to be the class of such functions:
\begin{equation}
\mathcal{L}=\{f:\mathbb{R}^{n}\rightarrow\mathbb{R}:f(x)\leq C|x|+D,f\, convex,C\geq0,D\in\mathbb{R}\}.\label{eq:Lelonclassdef}\end{equation}
This class is useful in our context since the intersection of currents
whose potentials belongs to $\mathcal{L}$ has finite total mass.
To see this, we first consider the case of the Monge-Ampère measure
of functions in $\mathcal{L}$ :
\begin{prop}
\label{pro:comparasionprop}Let $f\in\mathcal{L}$ so that we can
find a constant $c>0$ for which $f\leq c|x|$, when $|x|$ is sufficiently
large. Then $f$ satisfies\[
\int_{\mathbb{R}^{n}\times\mathbb{R}^{n}}(dd^{\#}f)^{n}<+\infty.\]
\end{prop}
\begin{proof}
Fix $\epsilon>0.$ For every $r>0$ we can find a constant $D>0$
such that $f\geq-D+(c+\epsilon)|x|$, if $|x|<r$ but $f\leq-D+(c+\epsilon)|x|,$
if $|x|>R$, for $R$ sufficiently large. Denote by $H$ the function
$\max\{f,-D+(c+\epsilon)|x|\}.$ Then $H$ is convex, and so $dd^{\#}H\geq0$.
We can exploit this as follows:\[
\int_{B(0,r)\times\mathbb{R}^{n}}(dd^{\#}f)^{n}=\int_{B(0,r)\times\mathbb{R}^{n}}(dd^{\#}H)^{n}\leq\int_{B(0,R)\times\mathbb{R}^{n}}(dd^{\#}H)^{n}=\]
\[
=\int_{\partial B(0,R)\times\mathbb{R}^{n}}d^{\#}(-D+(c+\epsilon)|x|)\wedge(dd^{\#}(-D+(c+\epsilon)|x|))^{n-1}=\]
\[
=\int_{B(0,R)\times\mathbb{R}^{n}}(dd^{\#}(-D+(c+\epsilon)|x|))^{n}\leq(c+\epsilon)^{n}\int_{\mathbb{R}^{n}\times\mathbb{R}^{n}}(dd^{\#}|x|)^{n}<+\infty.\]
Letting $r\rightarrow\infty$ we obtain\[
\int_{\mathbb{R}^{n}\times\mathbb{R}^{n}}(dd^{\#}f)^{n}<+\infty.\]
\end{proof}
\begin{prop}
\label{pro:intersectioninLelongclassgivesfinitemass}If $f_{1},...,f_{n}\in\mathcal{L},$then
\[
\int_{\mathbb{R}^{n}\times\mathbb{R}^{n}}dd^{\#}f_{1}\wedge...\wedge dd^{\#}f_{n}<+\infty.\]
\end{prop}
\begin{proof}
Since $f_{1}+...+f_{n}\in\mathcal{L}$ there exists a $C>0$ such
that $f_{1}+...+f_{n}\leq C|x|$, when $|x|$ is large enough. By
proposition \ref{pro:comparasionprop}, with $f=f_{1}+...+f_{n}$,
we obtain \[
\int_{\mathbb{R}^{n}\times\mathbb{R}^{n}}(dd^{\#}f_{1}+...+dd^{\#}f_{n})^{n}<+\infty.\]
 But $(dd^{\#}f_{1}+...+dd^{\#}f_{n})^{n}$ is a sum with one term
equal to $dd^{\#}f_{1}\wedge...\wedge dd^{\#}f_{n}$, and since every
term of the sum is a positive measure, we deduce that \[
\int_{\mathbb{R}^{n}\times\mathbb{R}^{n}}dd^{\#}f_{1}\wedge...\wedge dd^{\#}f_{n}<+\infty,\]
which concludes the proof.
\end{proof}
A slight modification of the proof of the above Proposition \ref{pro:comparasionprop}
gives us a useful comparison theorem, whose analogue in the complex
setting is well known.
\begin{prop}
\label{pro:comparasionprop-1}Let $f,g\in\mathcal{L}$ be such that
$f\leq g+O(1)$. Then\[
\int_{\mathbb{R}^{n}\times\mathbb{R}^{n}}(dd^{\#}f)^{n}\leq\int_{\mathbb{R}^{n}\times\mathbb{R}^{n}}(dd^{\#}g)^{n}.\]
\end{prop}
\begin{proof}
Fix $\epsilon>0$. For every $r>0$ we can find a constant $D>0$
such that $f\geq-D+g+\epsilon|x|$, if $|x|<r$ but $f\leq-D+g+\epsilon|x|,$
if $|x|>R$, for $R$ sufficiently large. Denote by $H$ the function
$\max\{f,-D+g+\epsilon|x|\}.$ Then $H$ is convex, and so $dd^{\#}H\geq0$.
As before:\[
\int_{B(0,r)\times\mathbb{R}^{n}}(dd^{\#}f)^{n}=\int_{B(0,r)\times\mathbb{R}^{n}}(dd^{\#}H)^{n}\leq\int_{B(0,R)\times\mathbb{R}^{n}}(dd^{\#}H)^{n}=\]
\[
=\int_{B(0,R)\times\mathbb{R}^{n}}(dd^{\#}(-D+g+\epsilon|x|)^{n}\leq\int_{\mathbb{R}^{n}\times\mathbb{R}^{n}}(dd^{\#}g+\epsilon dd^{\#}|x|)^{n}.\]
Letting $r\rightarrow\infty$ we obtain that, for every $\epsilon>0,$
\[
\int_{\mathbb{R}^{n}\times\mathbb{R}^{n}}(dd^{\#}f)^{n}\leq\int_{\mathbb{R}^{n}\times\mathbb{R}^{n}}(dd^{\#}g+\epsilon dd^{\#}|x|)^{n}.\]
This last integral contains terms of the type $(dd^{\#}g)^{n-k}\wedge\epsilon^{k}(dd^{\#}|x|)^{k}$
with $k=0,...,n$. If $k\neq n$, proposition \ref{pro:intersectioninLelongclassgivesfinitemass}
tells us that \[
\int_{\mathbb{R}^{n}\times\mathbb{R}^{n}}(dd^{\#}g)^{n-k}\wedge\epsilon^{k}(dd^{\#}|x|)^{k}\leq C_{k}\epsilon^{k}\]
and consequently there is a constant $C>0$ (independent of $\epsilon$)
for which \[
\int_{\mathbb{R}^{n}\times\mathbb{R}^{n}}(dd^{\#}g+\epsilon dd^{\#}|x|)^{n}\leq\int_{\mathbb{R}^{n}\times\mathbb{R}^{n}}(dd^{\#}g)^{n}+\epsilon C.\]
Letting $\epsilon\rightarrow0$ completes the proof.
\end{proof}
Interchanging the roles of $f$ and $g$ in the above proposition
gives us:
\begin{cor}
\label{cor:mongeamperecomparasion2}If $f,g\in\mathcal{L}$ satisfy
\[
|f-g|\leq C,\]
for some constant $C>0,$ then \[
\int_{\mathbb{R}^{n}\times\mathbb{R}^{n}}(dd^{\#}f)^{n}=\int_{\mathbb{R}^{n}\times\mathbb{R}^{n}}(dd^{\#}g)^{n}.\]

\end{cor}
In fact, the proof gives us a slightly stronger statement, which we
will find useful:
\begin{cor}
\label{cor:mongeamperecomparasion}If $f,g\in\mathcal{L}$ satisfy
\[
|f-g|\leq C+\epsilon|x|\]

for every $\epsilon>0$ and for some constant $C>0,$ then \[
\int_{\mathbb{R}^{n}\times\mathbb{R}^{n}}(dd^{\#}f)^{n}=\int_{\mathbb{R}^{n}\times\mathbb{R}^{n}}(dd^{\#}g)^{n}.\]
\end{cor}
\begin{example}
Let $K\subset\mathbb{R}^{n}$ be a convex set containing the origin,
and denote by $H_{K}$ its \emph{support function}, that is, $H_{K}(x)=\sup_{\xi\in K}\{x\cdot\xi\}$.
The polar of $K$, denoted $K^{\circ},$ is defined by $K^{\circ}=\{x:H_{K}(x)\leq1\}$.
If $\partial K$ is smooth and $K$ strictly convex, it is well known
that the map $x\mapsto\partial H_{K}(x)$ defines a diffeomorphism
between $\partial K$ and $\partial K^{\circ}$. We can thus introduce
$\frac{\partial H_{K}}{\partial x_{j}}$ as coordinates on $\partial K^{\circ}$
to obtain \[
n!Vol(K)=\int_{K\times\mathbb{R}^{n}}(dd^{\#}|x|^{2})^{n}=\int_{\partial K\times\mathbb{R}^{n}}d^{\#}|x|^{2}\wedge(dd^{\#}|x|^{2})^{n-1}=\]

\[
=c_{n}\int_{\partial K}\sum x_{i}\widehat{dx_{i}}=c_{n}\int_{\partial K^{\circ}}\sum\frac{\partial H_{K}}{\partial x_{i}}\cdot\widehat{d\frac{\partial H_{K}}{\partial x_{i}}}=\int_{K^{\circ}\times\mathbb{R}^{n}}d^{\#}H_{K}\wedge(dd^{\#}H_{K})^{n-1}=\]

\begin{equation}
=\int_{K^{\circ}\times\mathbb{R}^{n}}(dd^{\#}H_{K})^{n}=\int_{\mathbb{R}^{n}\times\mathbb{R}^{n}}(dd^{\#}H_{K})^{n}.\label{eq:volKequalsmaHK}\end{equation}
In the last equality we used proposition \ref{pro:homogenuousMAvanish}:
since $H_{K}$ is smooth outside the origin (thanks to $\partial K$
being smooth) and $1$-homogeneous, the support of $(dd^{\#}H_{K})^{n}$
is the origin. By approximation, the same formula holds without any
smoothness assumptions on $\partial K.$ 
\end{example}
One can generalize this example as follows: Recall that if $K_{1},...,K_{n}$
are convex sets in $\mathbb{R}^{n}$ one can define the \emph{mixed
volume }of $K_{1},...,K_{n}$, which we will denote by $V(K_{1},...,K_{n})$
as follows: consider the function \[
P(t_{1},...,t_{n}):=Vol(t_{1}K_{1}+...+t_{n}K_{n}),\]
 where \[
\sum_{j\in J}t_{j}K_{j}:=\{t_{j_{1}}\cdot x_{j_{1}}+...+t_{j_{l}}\cdot x_{j_{l}}:x_{j_{i}}\in K_{j_{i}},J=(j_{i},...,j_{l})\}\]
 is the Minkowski sum. As will follow from the proof of Proposition
\ref{pro:mixedvolumeMA-1}, $P$ is a $n-$homogeneous polynomial
in $n$ variables:\[
P(t_{1},...,t_{n})=\sum_{i_{1},...,i_{n}=1}^{n}a_{i_{1},...,i_{n}}\cdot t_{i_{1}}\cdot...\cdot t_{i_{n}},\]
 for some coefficients $a_{i_{1},...,i_{n}}$. The mixed volume is
the coefficient in this polynomial corresponding to the monomial $t_{1}\cdot....\cdot t_{n}$,
that is, \[
V(K_{1},...,K_{n}):=a_{1,...,n}.\]
We claim the following:
\begin{prop}
\label{pro:mixedvolumeMA-1}Let $K_{1},...,K_{n}$ be convex sets
in $\mathbb{R}^{n}$ with corresponding support functions $H_{K_{i}}.$
Then \[
dd^{\#}H_{K_{1}}\wedge...\wedge dd^{\#}H_{K_{n}}=n!\cdot V(K_{1},...,K_{n})\delta_{0}\omega_{n},\]
and $P(t_{1},...,t_{n}):=Vol(t_{1}K_{1}+...+t_{n}K_{n})$ is a $n-$homogeneous
polynomial.\end{prop}
\begin{proof}
To begin with, we note that \begin{equation}
H_{tK+sL}=tH_{tK}+sH_{L}\label{eq:minkowskiisumsupportfunct}\end{equation}
 if $K,L$ are compact subsets, and $t,s$ real numbers. This is justified
by the following equalities:

\[
H_{tK+sL}(x)=\sup_{\xi=\xi_{1}+\xi_{2}\in tK+sL}\{\xi\cdot x\}=\sup_{\xi_{1}\in K,\xi_{2}\in L}\{t\xi_{1}\cdot x+s\xi_{2}\cdot x\}=tH_{K}(x)+sH_{L}(x).\]
Generalizing this slightly, we obtain the identity\[
MA(H_{t_{1}K_{1}+...+t_{n}K_{n}})=MA(t_{1}H_{K_{1}}+...+t_{n}H_{K_{n}}),\]
and thus $MA(H_{t_{1}K_{1}+...+t_{n}K_{n}})$ is a $n-$homogeneous
polynomial in $(t_{1},..,t_{n}).$ Moreover, by \eqref{eq:volKequalsmaHK}
we know that\[
MA(H_{t_{1}K_{1}+...+t_{n}K_{n}})=n!Vol(t_{1}K_{1}+...+t_{n}K_{n})\cdot\delta_{0}\cdot\omega_{n}.\]
This immediately gives us that $P(t_{1},...,t_{n})=Vol(t_{1}K_{1}+...+t_{n}K_{n})$
is an $n$-homogeneous polynomial. Moreover, comparing coefficients
of the two polynomials, we see that\[
dd^{\#}H_{K_{1}}\wedge...\wedge dd^{\#}H_{K_{n}}=n!\cdot V(K_{1},...,K_{n})\cdot\delta_{0}\cdot\omega_{n}\]
as desired.\end{proof}
\begin{rem}
These results should be compared with the results already obtained
in \cite{Passare}.
\end{rem}
Let $f$ be a convex function on $\mathbb{R}^{n}$, belonging to the
Lelong class $\mathcal{L}$. To this $f$ we associate the function
\[
\tilde{f}(x)=\lim_{t\rightarrow\infty}\frac{{f(tx)}}{t},\]
where the limits exists thanks to the assumption on linear growth
at infinity. This function $\tilde{f}$ is easily seen to be convex
and one-homogeneous. Moreover, we claim that \begin{equation}
|f-\tilde{f}|\leq C+\epsilon|x|,\label{eq:homogenizationinequality}\end{equation}
 for every $\epsilon>0$ . This is readily seen as follows: By standard
properties of convex functions, the expression $\frac{f(tx)-f(0)}{t}$
is increasing in $t$, for every $x$. We recall Dini's theorem which
says the following: if $f$ is continuous on a compact set $K$, and
$f_{t}$ is a monotone sequence of continuous functions which converge
to $\tilde{f}$ pointwise, then the convergence is in fact uniform
on $K$. Consequently, for each $\epsilon>0$ we have that\[
\sup_{|x|=1}|\frac{f(tx)-f(0)}{t}-\tilde{f}(x)|<\epsilon,\]
if $t>T$, for some $T>0$ . Multiplying through by $t$ we obtain
\[
\sup_{|x|=1}|f(tx)-f(0)-\tilde{f}(tx)|<t\epsilon,\]
if $t$ is sufficiently large, which implies \eqref{eq:homogenizationinequality}.
An application of Corollary \ref{cor:mongeamperecomparasion} shows
that the total Monge-Ampère mass of $f$ equals that of $\tilde{f}$
:
\begin{prop}
\label{pro:relationbetweenfanditshomogenization}With $f$ and $\tilde{f}$
as above \[
\int_{\mathbb{R}^{n}\times\mathbb{R}^{n}}(dd^{\#}f)^{n}=\int_{\mathbb{R}^{n}\times\mathbb{R}^{n}}(dd^{\#}\tilde{f})^{n}.\]

\end{prop}
As we will see in Section \ref{sub:Bezout's-theorem}, the above Proposition
is essentially Bezout's theorem in tropical geometry.

\section{Lelong numbers, trace measures and push forwards of currents.}
\begin{defn}
The\emph{ trace measure} of a $(p,p)$-current $T$ is defined as\[
\Theta_{T}(U)=\frac{1}{2^{n-p}(n-p)!}\int_{U\times\mathbb{R}^{n}}T\wedge(dd^{\#}|x|^{2})^{n-p},\]
 for each Borel-set $U\subset\mathbb{R}^{n}.$ \end{defn}
\begin{prop}
\label{pro:tracemeasureinequality}If  $T$ is a positive $(p,p)-$current,
then $\Theta_{T}$ is a positive measure, and \begin{equation}
|T_{IJ}|\leq C\cdot\Theta_{T},\label{eq:tracemeasureineq}\end{equation}
for some $C>0$. \end{prop}
\begin{proof}
This follows immediately from Proposition \ref{pro:tracemeasureinequality-1}.\end{proof}
\begin{rem}
Let us compare with the complex setting: if $S$ is a complex, weakly
positive $(p,p)$-current, then $S$ always satisfies a trace measure
inequality of the type \eqref{eq:tracemeasureineq}. However, in our
setting the form $\alpha$, on $\mathbb{E}=\mathbb{R}^{4}\times\mathbb{R}^{4}$,
given by \[
\alpha=dx_{1}\wedge dx_{2}\wedge d\xi_{3}\wedge d\xi_{4}+dx_{2}\wedge dx_{3}\wedge d\xi_{1}\wedge d\xi_{4}-dx_{1}\wedge dx_{3}\wedge d\xi_{2}\wedge d\xi_{4}+\]
\[
+dx_{3}\wedge dx_{4}\wedge d\xi_{1}\wedge d\xi_{2}+dx_{1}\wedge dx_{4}\wedge d\xi_{2}\wedge d\xi_{3}-dx_{2}\wedge dx_{4}\wedge d\xi_{1}\wedge d\xi_{3},\]
 satisfies \begin{equation}
\alpha\wedge v\wedge J(v)=0\label{eq:notpositive}\end{equation}
 for every $(1,0)-$form $v$. Thus, $\alpha$ is weakly positive,
but all the diagonal elements are 0, and consequently $\alpha$ does
not satisfy an inequality of the type $\eqref{eq:tracemeasureineq}$.
This implies a significant difference between our setting and the
complex setting. In fact, in the complex case, the strongly positive
forms constitute a basis for the space of all forms. Equation \eqref{eq:notpositive}
tells us that this is not the case in our setting. \end{rem}
\begin{prop}
\label{pro:increasingquotient}For fixed $x\in\mathbb{R}^{n}$, if
$T$ is a weakly positive $(p,p)$-current, then the function \[
r\mapsto\frac{{\Theta_{T}}}{r^{n-p}}(B(x,r))\]
is increasing on $\mathbb{R}_{+}$. \end{prop}
\begin{proof}
We can assume that $x=0$. By equation \eqref{eq:ddshaprabsolutevalueofx}
we have, \[
dd^{\#}|x|=\frac{dd^{\#}|x|^{2}}{2|x|}-\frac{d|x|^{2}\wedge d^{\#}|x|^{2}}{4|x|^{3}},\]
and $\frac{d|x|^{2}\wedge d^{\#}|x|^{2}}{4|x|^{3}}=0$ on the sphere
$|x|=r$. Moreover, $d^{\#}|x|=\frac{d^{\#}|x|^{2}}{2|x|}$, so that
$d^{\#}|x|=\frac{d^{\#}|x|^{2}}{2r}$ if $|x|=r$. Combining these
observations, we find that \[
\int_{\{|x|=r\}\times\mathbb{R}^{n}}T\wedge d^{\#}|x|\wedge(dd^{\#}|x|)^{n-p-1}=\int_{\{|x|=r\}\times\mathbb{R}^{n}}T\wedge\frac{d^{\#}|x|^{2}}{2|x|}\wedge(\frac{dd^{\#}|x|^{2}}{2|x|})^{n-p-1}=\]
\[
=\frac{{1}}{(2r)^{n-p}}\int_{\{|x|=r\}\times\mathbb{R}^{n}}T\wedge d^{\#}|x|^{2}\wedge(dd^{\#}|x|^{2})^{n-p-1}.\]
Thus, by Stokes' theorem we obtain \[
\int_{B(0,r)\times\mathbb{R}^{n}}T\wedge(dd^{\#}|x|^{2})^{n-p}=\int_{\{|x|=r\}\times\mathbb{R}^{n}}T\wedge d^{\#}|x|^{2}\wedge(dd^{\#}|x|^{2})^{n-p-1}=\]
\[
=(2r)^{n-p}\int_{\{|x|=r\}\times\mathbb{R}^{n}}T\wedge d^{\#}|x|\wedge(dd^{\#}|x|)^{n-p-1}=\]
\[
=(2r)^{n-p}\int_{B(0,r)\times\mathbb{R}^{n}}T\wedge(dd^{\#}|x|)^{n-p}.\]
Since $T$ is weakly positive, $T\wedge(dd^{\#}|x|)^{n-p}$ is a positive
measure. Thus the function \[
r\mapsto\frac{{\Theta_{T}}}{r^{n-p}}(B(0,r))\]
is increasing. \end{proof}
\begin{cor}
Let $T$ be a clo\label{cor:supportthm}sed, positive $(p,p)-$current.
If $K$ is a compact set with vanishing $(n-p)$-dimensional Hausdorff
measure, then $T$ vanishes on $K$. In particular, if $Supp(T)$
has vanishing $(n-p)$-dimensional Hausdorff measure, then $T=0.$\end{cor}
\begin{proof}
Let $K$ be a compact set, satisfying the assumptions of the hypothesis.
The condition $\mathcal{H}^{n-p}(K)=0$, means that we can, for every
$\epsilon>0,$ find a finite number of balls $B(a_{j},r_{j})$ for
which $K\subset\cup_{j}B(a_{j},r_{j})$ and \[
\sum r_{j}^{n-p}\leq\epsilon.\]
We can assume that each $r_{j}\leq1.$ By Proposition \ref{pro:increasingquotient}
we see that \[
\frac{{\Theta_{T}}}{r_{j}^{n-p}}(B(a_{j},r_{j}))\leq\Theta_{T}(B(a_{j},1))\leq\Theta_{T}(K^{'}),\]
 where $K^{'}$ is a compact set such that $K\subset\cup_{j}B(a_{j},1)\subset K^{'}$
and thus, with $C=\Theta_{T}(K^{'})$ we obtain the inequality: $\Theta_{T}(B(a_{j},r_{j}))\leq Cr_{j}^{n-p}$
for all $j.$ We conclude that \[
\Theta_{T}(K)\leq\sum_{j}\Theta_{T}((B(a_{j},r_{j}))\leq C\sum_{j}r_{j}^{n-p}\leq C\epsilon,\]
 and thus, $T_{|K}=0$, since $|T_{IJ}|(K)\leq\Theta_{T}(K)$ by Proposition
\ref{pro:tracemeasureinequality}.
\end{proof}
As a consequence of the proposition, we can define the \emph{Lelong
number} of a weakly positive, closed $(p,p)-$current $T$ at a point
$x$ by \[
\nu_{x}(T)=\lim_{r\rightarrow0}\frac{{\Theta_{T}(B(x,r))}}{Vol(B^{n-p})r^{n-p}},\]
where $Vol(B^{n-p})$ is the volume of the $(n-p)-$dimensional unit
ball. We define the Lelong number of a convex function $f:\mathbb{R}^{n}\rightarrow\mathbb{R}$
by \[
\nu_{x}(f)=\nu_{x}(dd^{\#}f).\]

\begin{example}
We calculate the Lelong number of the function $x\mapsto|x|$. We
begin with considering the behaviour at the origin. By \eqref{eq:ddshaprabsolutevalueofx}
and Stokes', \[
\int_{B(0,r)\times\mathbb{R}^{n}}dd^{\#}|x|\wedge(dd^{\#}|x|^{2})^{n-1}=\]
\[
=\frac{1}{2r}\int_{B(0,r)\times\mathbb{R}^{n}}(dd^{\#}|x|^{2})^{n}=2^{n-1}\cdot n!\cdot r^{n-1}\cdot Vol_{n}(B(0,1)),\]
and thus\[
\nu_{0}(|x|)=\lim_{r\rightarrow0}\frac{\Theta(B(0,r))}{Vol_{n-1}(B(0,r))}=n\cdot\frac{Vol_{n}(B(0,1))}{Vol_{n-1}(B(0,1))}.\]
On the other hand, at a point $x_{0}$ away from the origin our function
is smooth, and so the form $dd^{\#}|x|$ is locally smooth. But if
$g$ is a smooth function in a neighbourhood around $x$, then the
trace measure of $dd^{\#}g(x)$ is just the Laplacian of $g$ at $x$
and thus there is a constant $C>0$ such that \[
\Theta_{dd^{\#}g}(B(x,r))\leq Cr^{n},\]
 since every coefficient $\frac{\partial^{2}g}{\partial x_{i}\partial x_{j}}$
is uniformly bounded in some neighbourhood around $x$. Thus\[
\nu_{x}(g)\leq\lim_{r\rightarrow0}\frac{Cr^{n}}{r^{n-1}Vol_{n-1}(B(0,1))}=0,\]

and so $\nu_{x_{0}}(|x|)=0$ if $x_{0}\neq0$. 
\end{example}
The above argument displays the following expected behaviour of the
Lelong number, showing that the Lelong number is a measurement of
the singularity at a point:
\begin{prop}
If a $(1,1)-$current $T$ is locally smooth around a point $x,$
then $\nu(T,x)=0$. \end{prop}
\begin{rem}
In complex analysis, it is a well known theorem due to Y.T. Siu, which
states that the set $\{x:\nu(T,x)\geq c\}$ constitutes an analytic
variety, for each $c>0$ (in this remark, $\nu(T,x)$ denotes the
complex version of the Lelong number). It would be interesting to
know if there is a corresponding result in our setting. One could
perhaps hope that one would obtain tropical varieties (which we define
in Section \ref{sec:Tropical-geometry}), but this is not the case
as the example $f=\max(|x|,1)$ shows: $\{x\in\mathbb{V}:\nu(dd^{\#}\max(|x|,1),x)\geq c\}=\{x:|x|=1\}$
which is not a tropical variety.
\end{rem}

\subsection{Push forwards of currents}

Let $f:\mathbb{V}\rightarrow\mathbb{V}^{'}$ be an affine map, with
$dim(\mathbb{V})=dim(\mathbb{V}^{'})=n$, inducing a map $\tilde{f}:\mathbb{E}\rightarrow\mathbb{E}^{'}$.
Then we can define the push-forward $f_{*}T$ of a current $T\in\mathcal{D}_{n-p,n-p}(\mathbb{E})$
via the formula \begin{equation}
\left\langle f_{*}T,\alpha\right\rangle =\left\langle T,f^{*}\alpha\right\rangle ,\label{eq:pushforwardpullback}\end{equation}
where $\alpha\in\mathcal{D}^{p,p}$. This formula only makes sense
if $f^{*}\alpha$ has compact support on $Supp(T)$, and so we first
demand that $f$ is such that $f^{-1}(K)\bigcap Supp(T)$ is compact
for every compact $K\subset\mathbb{V}',$ or in other words, the restriction
of $f$ to $Supp(T)$ is proper. For such $f$ the induced map $f^{*}:\mathcal{E}^{p,p}(\mathbb{E}')\rightarrow\mathcal{E}^{p,p}(\mathbb{E})$
is continuous and thus the above formula defines an element in $\mathcal{D}_{n-p,n-p}(\mathbb{E}')$
. If $T$ is weakly positive, then \[
\left\langle f_{*}T,\alpha_{1}\wedge J(\alpha_{1})\wedge...\wedge\alpha_{p}\wedge J(\alpha_{p})\right\rangle =\]
\[
=\left\langle T,f^{*}\alpha_{1}\wedge J(f^{*}\alpha_{1})\wedge....\wedge f^{*}\alpha_{p}\wedge J(f^{*}\alpha_{p})\right\rangle \geq0.\]
Thus we have the following proposition.
\begin{prop}
If \textup{$T\in\mathcal{D}_{n-p,n-p}(\mathbb{E})$ is weakly positive
and $f:\mathbb{V}\rightarrow\mathbb{V}'$ is an (non-constant) affine
function, then $\tilde{f}_{*}T$ is a weakly positive current in $\mathcal{D}_{n-p,n-p}(\mathbb{E}')$.}\end{prop}
\begin{example}
Every form $\beta\in\mathcal{E}^{p,p}$ can be considered as a current
acting on compactly supported forms of complementary degree. Now,
if $f:\mathbb{V}\rightarrow\mathbb{V}$ is a non-singular affine map,
then \[
\left\langle \alpha,f_{*}\beta\right\rangle =\left\langle f^{*}\alpha,\beta\right\rangle =\int_{\mathbb{E}}f^{*}\alpha\wedge\beta=\frac{1}{|det(f^{-1})|}\int_{\mathbb{E}}(f^{-1})^{*}(f^{*}\alpha\wedge\beta)=\]
\[
=|det(f)|\int_{\mathbb{E}}\alpha\wedge(f^{-1})^{*}\beta=|det(f)|\left\langle \alpha,(f^{-1})^{*}\beta\right\rangle ,\]
where we used formula \eqref{eq:variablechangeformula} in the third
equality. Thus we see that, \begin{equation}
f_{*}\beta=|detf|(f^{-1})^{*}\beta,\beta\in\mathcal{E}^{p,p}.\label{eq:pullbackpushforwardchangeofvariable}\end{equation}

\end{example}
Now let us consider the projection (here we write $\mathbb{V},\mathbb{W}=\mathbb{R}^{n}$)
\[
\pi:\mathbb{R}^{n}\rightarrow\mathbb{R}^{n-k},\pi(x_{1},...,x_{n})=(x_{1},...,x_{n-k}),\]
where $k\geq0$ and take a ($p,p)-$form $\alpha$ on $\mathbb{R}^{n}$,
with locally integrable coefficients, which we will consider as a
current. Assume this form $\alpha$ is such that $\pi$ is proper
on its support, and let $\omega$ be a $(n-p,n-p)-$form on $\mathbb{R}^{n-k}=\{(x_{1},...,x_{n-k})\}$
. Observe that if $n-p>n-k,$ the form $\omega$ is $0$, so we assume
this is not the case. We regard, for each $x\in\mathbb{R}^{n-k}$,
the set $\pi^{-1}(x)$ as $\mathbb{R}^{k}$ with coordinates $(x_{n-k+1},...,x_{n})$.
Thus $\omega$ only contains differentials $dx_{i}$ and $d\xi_{i}$
with $1\leq i\leq n-k$, and $\pi^{*}\omega=\omega$. Let us write
$\alpha=\sum_{|I|=|J|=p}\alpha_{IJ}dx_{I}\wedge d\xi_{J}$, and $\omega=\sum_{|K|=|L|=n-p}\omega_{KL}dx_{K}\wedge d\xi_{L}$
where each $K$ and $L$ only contain indices between $1$ and $n-k$.
The $(n,n)-$form $\alpha\wedge\pi^{*}\omega$ is a sum of terms of
the form $\alpha_{IJ}\omega_{KL}dx_{I}\wedge d\xi_{J}\wedge dx_{K}\wedge d\xi_{L}$
and such a term vanishes if $I$ and $K$, or if $J$ and $L$, contain
the same indices. Since $n-p\leq n-k$, this implies that the only
terms in the expression defining $\alpha$ that will contribute to
the push-forward are those for which $dx_{I}\wedge d\xi_{J}$ contains
the differential $dx_{n-k+1}\wedge...\wedge dx_{n}\wedge d\xi_{n-k+1}\wedge...\wedge d\xi_{n}$.
In effect, we can write $\alpha$ as $\alpha=\left\{ \sum_{|I'|=|J'|=p-k}\alpha_{I'J'}dx_{I'}\wedge d\xi_{J'}\right\} \wedge dx_{n-k+1}\wedge...\wedge dx_{n}\wedge d\xi_{n-k+1}\wedge...\wedge d\xi_{n}+R$,
where $R$ is such that $R\wedge\omega'=0$ for every $(n-p,n-p)$-form
$\omega'$ on $\mathbb{R}^{n-k}$, and hence will not contribute to
the push forward. The definition of push-forward tells us that\[
\left\langle \pi_{*}\alpha,\omega\right\rangle =\left\langle \alpha,\pi^{*}\omega\right\rangle =\]
\[
=\int_{\mathbb{W}}\int_{(x_{1},...,x_{n-k})\in\mathbb{R}^{n-k}}\int_{(x_{n-k+1},...,x_{n})\in\pi^{-1}(x)}\alpha(x_{1},...,x_{n})\wedge\omega(x_{1},...,x_{n-k}),\]
and from this we deduce that the push forward of $\alpha$ under $\pi$
is given by the $(p-k,p-k)$-form\[
(\pi_{*}\alpha)(x_{1},...,x_{n-k})=\]
\[
=\int_{\pi^{-1}(x_{1},...,x_{n-k})\times\mathbb{W}}\left\{ \sum_{|I'|=|J'|=p-k}\alpha_{I'J'}dx_{I'}\wedge d\xi_{J'}\right\} \wedge dx_{n-k+1}\wedge...\wedge d\xi_{n}=\]
\[
=\sum_{|I'|=|J'|=p-k}\tilde{\alpha}_{I'J'}(x_{1},...,x_{n-k})dx_{I'}\wedge d\xi_{J'},\]
where \begin{equation}
\tilde{\alpha}_{I'J'}(x_{1},...,x_{n-k})=\int_{\pi^{-1}(x_{1},...,x_{n-k})}\alpha_{I'J'}(x_{1},...,x_{n})dV_{k}(x_{n-k+1},...,x_{n})\label{eq:pushforwardintergraltoconverge}\end{equation}
and $dV_{k}$ is the volume measure induced by the chosen inner product
on $\mathbb{R}^{n}$. We have hence obtained an explicit formula for
the push-forward of a form $\alpha$ under a projection. By using
an approximation argument, the discussion above still holds true if
we assume $\alpha$ to be a strongly positive \emph{current}. Now,
since $\pi^{*}d\omega=d\pi^{*}\omega$, we see that \[
\left\langle d\pi_{*}\alpha,\omega\right\rangle =\pm\left\langle \pi_{*}\alpha,d\omega\right\rangle =\pm\left\langle \alpha,\pi^{*}d\omega\right\rangle =\pm\left\langle \pi_{*}\alpha,d\omega\right\rangle =\left\langle \pi_{*}d\alpha,\omega\right\rangle ,\]
which implies that $\pi_{*}d\alpha=d\pi_{*}\alpha$. Thus, if $\alpha$
is closed, then $\pi_{*}\alpha$ is closed as well. Since strong positivity
of forms is preserved under pullbacks, we find that if $\alpha$ is
a weakly positive form, then $\pi_{*}\alpha$ is weakly positive.
Moreover, the formula \[
\left\langle f_{*}\alpha,\sigma_{(n-k)-(p-k)}\beta\wedge J(\beta)\right\rangle =\left\langle \alpha,\sigma_{n-p}f^{*}\beta\wedge J(f^{*}\beta)\right\rangle ,\]
tells us that if $\alpha$ is positive, then so is $f_{*}\alpha$.
Thus we have:
\begin{prop}
\label{pro:pushforwardpropersupport}Let $\pi$ be a projection from
$\mathbb{R}^{n}$ onto a $(n-k)$-dimensional subspace, and let $\alpha$
be a (weakly) positive $(p,p)$-current on $\mathbb{R}^{n}\times\mathbb{R}^{n}$
such that $\pi$ is proper on the support of $\alpha$. Then the push-forward
$\pi_{*}\alpha$ is a well-defined $(p-k,p-k)$-current. Moreover,
$\pi_{*}\alpha$ is (weakly) positive, and if $\alpha$ is closed
then $\pi_{*}\alpha$ is closed as well. 
\end{prop}
Now let $f_{1},f_{2},...,f_{p}\in\mathcal{L}$ and let $S=dd^{\#}f_{1}\wedge...\wedge dd^{\#}f_{p}$.
Let $\pi$ be a projection from $\mathbb{R}^{n}$ onto a $(n-k)$-dimensional
subspace $L$ as above. Since the restriction of $\pi$ to $Supp(S)$
might not be proper, the expression $\pi_{*}S$ has no meaning as
of yet. However, if $\chi$ is a continuous function with compact
support on $\pi^{-1}(L)$ and with values in $[0,1]$, we can consider
the positive (but not closed) current $\pi_{*}\chi S$ on $L$. We
have the following lemma:
\begin{lem}
\label{lem:pushforwardlemma}For every compact subset $K\subset L$
there exists a constant $C_{K}>0$ which is independent of $\chi$,
such that the measure-coefficients of $\pi_{*}\chi S$ applied to
$K$ are bounded by $C_{K}$ .\end{lem}
\begin{proof}
We choose coordinates so that $L=\mathbb{R}^{n-k}$ as above and write
$x'=(x_{1},...,x_{n-k})$. It is enough to prove that the statement
holds for every ball with center at the origin, in $\mathbb{R}^{n-k}$.
Fix such a ball $B(0,R)$. Define a function $\phi$ on $\mathbb{R}^{n-k}$
by letting $\phi(x')=|x'|^{2}$ for $x'\in B(0,2R)$, and $\phi(x')=4R|x'|-4R^{2}$
otherwise. Then $\phi\in\mathcal{L}$, and $dd^{\#}\phi=\sum_{i=1}^{n-k}dx_{i}\wedge d\xi_{i}$
on $B(0,R)$. Thus, $(dd^{\#}\phi)^{n-p}=(dd^{\#}|x'|^{2})^{n-p}$
on $B(0,R)$, which implies that\[
\Theta_{B(0,R)}(\pi_{*}\chi S)=\]
\[
=\int_{(B(0,R)\times\mathbb{R}^{k})\times\mathbb{R}^{n}}\chi S\wedge(dd^{c}|x'|^{2})^{n-p}=\int_{(B(0,R)\times\mathbb{R}^{k})\times\mathbb{R}^{n}}\chi S\wedge(dd^{c}\phi)^{n-p}\leq\]
\[
\leq\sup\chi\int_{\mathbb{R}^{n}\times\mathbb{R}^{n}}S\wedge(dd^{c}\phi)^{n-p}.\]
By Proposition \ref{pro:intersectioninLelongclassgivesfinitemass}
the right hand side is finite. Since $\pi_{*}\chi S$ is positive,
we can apply Proposition \ref{pro:tracemeasureinequality} to obtain
that every coefficient of $\pi_{*}\chi S$ when applied to $B(0,R)$
has mass bounded by the trace-measure of $\chi\pi_{*}S$ acting on
$B(0,R)$ and is consequently less than some constant $C_{R}>0$ depending
on $R$ (since $\phi$ depends on $R$) but not on $\chi$. The proposition
follows. 
\end{proof}
Assume that the functions $f_{i}$ are smooth. Then Lemma \ref{lem:pushforwardlemma}
together with \eqref{eq:pushforwardintergraltoconverge} tells us
that there is a constant $C>0$ such that for every positive, compactly
supported, continuous function $\chi$ defined on $\pi^{-1}(L)=\mathbb{R}^{k}$
with values in $[0,1]$, the following inequality holds: \[
|\int_{\pi^{-1}(x_{1},...,x_{n-k})}\chi(x_{n-k+1},...,x_{n})S_{I'J'}(x_{1},...,x_{n})dV_{k}(x_{n-k+1},...,x_{n})|\leq\]
\[
\leq C\cdot Sup(\chi).\]
 This implies that all of the integrals \[
\int_{\pi^{-1}(x_{1},...,x_{n-k})}S_{I'J'}(x_{1},...,x_{n})dV_{k}(x_{n-k+1},...,x_{n})\]
 converge. Thus, if we let $\chi_{R}$ be functions of the kind considered
with the additional assumption that their support should exhaust $\pi^{-1}(L)$
as $R\rightarrow\infty,$ we see that the weak limit of $\pi_{*}\chi_{R}S$
as $R\rightarrow\infty$ exists, and we put\[
\pi_{*}dd^{\#}f_{1}\wedge...\wedge dd^{\#}f_{p}=\lim_{R\rightarrow\infty}\pi_{*}\chi_{R}S.\]
It is easy to see that it does not depend on the choice of functions
$\chi_{R}$, and if $\pi$ were to have proper support on $dd^{\#}f_{1}\wedge...\wedge dd^{\#}f_{p}$,
this definition would coincide with the previous one given above .
We want to show that this current is closed and positive. For this
we construct explicit choices of $\chi_{R}$ as follows:

Let $\tilde{\chi}_{R}:\mathbb{R}_{+}\rightarrow\mathbb{R}$ be the
piecewise linear function, equal to 1 on $[0,R]$, equal to $0$ on
$[2R,+\infty)$ and linear in between. Then $\tilde{\chi}_{R}^{'}=R^{-1}$
on the interval $[R,2R]$ and 0 otherwise. We now put $\chi_{R}(x)=\tilde{\chi}_{R}(|x|)$. 
\begin{prop}
\label{pro:pushforwardofcurrentsinLelongclass}Let $\pi:\mathbb{R}^{n}\rightarrow L$
be a projection, where $L$ is a $(n-k)-$dimensional subspace of
$\mathbb{R}^{n}$, and assume that $f_{1},f_{2},...,f_{p}\in\mathcal{L}$.
Then \[
\pi_{*}dd^{\#}f_{1}\wedge...\wedge dd^{\#}f_{p}\]
is a well-defined positive, closed $(p-k,p-k)-$current. If $p<k$
then $\pi_{*}dd^{\#}f_{1}\wedge...\wedge dd^{\#}f_{p}=0$.\end{prop}
\begin{proof}
We have yet to show that it is closed and positive. Let $S=dd^{\#}f_{1}\wedge...\wedge dd^{\#}f_{p}$
as above. Positivity is clear, since, if $\alpha\in\mathcal{D}^{n-p,0}$,
\[
\left\langle (\pi_{*}S),\sigma_{(n-k)-(p-k)}\alpha\wedge J(\alpha)\right\rangle =\lim_{R\rightarrow\infty}\int_{\mathbb{E}}\chi_{R}S\wedge\sigma_{n-p}\cdot\pi^{*}\alpha\wedge\pi^{*}J(\alpha)\geq0,\]
by the positivity of $S$. To prove closedness, we use the specific
function $\chi_{R}$ constructed above. Then, for $\alpha\in\mathcal{D}^{n-p-1,n-p}(L),$
\[
\left\langle d(\pi_{*}S),\alpha\right\rangle =\pm\lim_{R\rightarrow\infty}\left\langle \chi_{R}\cdot S,\pi^{*}(d\alpha)\right\rangle =\lim_{R\rightarrow\infty}\left\langle d\chi_{R}\wedge S,\pi^{*}\alpha\right\rangle \]
 thanks to $S$ being closed. We need to show that $\lim_{R\rightarrow\infty}\left\langle d\chi_{R}\wedge S,\pi^{*}\alpha\right\rangle =0$.
To this end we define the following bi-linear form: \[
(v,w)=\int_{\mathbb{E}}S\wedge\sigma_{n-p}v\wedge J(w),\]
where $v,w$ are compactly supported, smooth $(n-p,0)-$forms on $\mathbb{R}^{n}$.
Clearly $(v,v)\geq0$, since $S$ is positive, and thus the bi-linear
form is positive definite. A variant of the Cauchy-Schwartz inequality
tells us that for each $\epsilon>0$, \begin{equation}
(v,w)\leq\epsilon(v,v)+\epsilon^{-1}(w,w).\label{eq:CSinequality}\end{equation}
 Let us define, for each $R>0$, \[
\psi_{R}(t)=\begin{cases}
\frac{2}{R}t-\frac{3}{2}, & \text{if }t\in[0,R)\\
\frac{t^{2}}{2R^{2}}, & \text{if }t\in[R,2R]\\
\frac{4}{R}t-6, & \text{if }t\in(2R,\infty)\end{cases}.\]
Then $\psi_{R}$ is convex, $\psi_{R}(|x|)\in\mathcal{L}$, and $\psi_{R}^{''}(t)=1/R^{2}$
if $t\in[R,2R]$ and 0 otherwise. Moreover, a direct calculation shows
that \begin{equation}
dd^{\#}\psi_{R}-d\chi_{R}\wedge d^{\#}\chi_{R}\geq0.\label{eq:psirlargerthanchir}\end{equation}
 Let $v$ be a $(n-p-1,0)$-form, and $w$ a $(0,n-p)$-form defined
on $\mathbb{R}^{n-k},$ both smooth and with compact support. Then
\begin{equation}
|\left\langle d\chi_{R}\wedge S,\pi^{*}v\wedge\pi^{*}w\right\rangle |=|\int_{\mathbb{E}}S\wedge d\chi_{R}\wedge\pi^{*}v\wedge\pi^{*}w|=|(d\chi_{R}\wedge\pi^{*}v,\pi^{*}w)|\label{eq:innerproductintegral}\end{equation}
which by \eqref{eq:CSinequality} is dominated by \[
\epsilon\int_{\mathbb{E}}S\wedge\pi^{*}w\wedge J(\pi^{*}w)+\epsilon^{-1}\int_{\mathbb{E}}S\wedge d\chi_{R}\wedge\pi^{*}v\wedge J(d\chi_{R}\wedge\pi^{*}v)=I+II.\]
Here the form $\pi^{*}w$ will not have compact support on $\mathbb{R}^{n}$
so the first term actually has no meaning. However, we may here replace
$\pi^{*}w$ with $\chi_{3R}\pi^{*}w$ which has compact support on
$\mathbb{R}^{n}$; doing so will not affect \eqref{eq:innerproductintegral},
since $\chi_{3R}=1$ on $Supp(d\chi_{R})$. Using Lemma \ref{lem:pushforwardlemma}
we see that the first term, $I$, is bounded by $\epsilon C_{w}.$
For the second term, $II$, we show that the trace measure of the
strongly positive current, $S\wedge d\chi_{R}\wedge J(d\chi_{R})$
tends to $0$ as $R\rightarrow\infty:$ For each multi-index $I$
of length $n-p$, we can use the idea of Lemma \ref{lem:pushforwardlemma}
to find a function $\phi_{I}\in\mathcal{L},$ such that, \[
dd^{\#}\phi_{I}=dd^{\#}\sqrt{x_{i_{1}}^{2}+...+x_{i_{n-p}}^{2}}=dx_{I}\wedge d\xi_{I},\]
 on $B(0,2R)$. Thus, since $S\wedge d\chi_{R}\wedge J(d\chi_{R})$
is strongly positive, \[
|\int_{B(0,2R)\times\mathbb{R}^{n}}S\wedge d\chi_{R}\wedge J(d\chi_{R})\wedge dx_{I}\wedge d\xi_{I}|\leq|\int_{\mathbb{R}^{n}\times\mathbb{R}^{n}}S\wedge d\chi_{R}\wedge J(d\chi_{R})\wedge dd^{\#}\phi_{I}|.\]
By \eqref{eq:psirlargerthanchir} the last integral is dominated by
\[
|\int_{\mathbb{E}}S\wedge dd^{\#}\psi_{R}\wedge dd^{\#}\phi_{I}|.\]
Thus, since $dd^{\#}\psi_{R}(|x|)=R^{-2}dd^{\#}|x|$ on $B(0,R)$
and zero otherwise, we obtain by Proposition \ref{pro:intersectioninLelongclassgivesfinitemass}
that \[
\epsilon^{-1}|\int_{B(0,R)\times\mathbb{R}^{n}}S\wedge d\chi_{R}\wedge J(d\chi_{R})\wedge dx_{I}\wedge d\xi_{I}|\leq\]
\[
\leq\epsilon^{-1}R^{-2}\int_{\mathbb{E}}S\wedge dd^{\#}|x|\wedge dd^{\#}\phi_{I}\leq\epsilon^{-1}R^{-2}D,\]
for some constant $D>0$ independent of $R$. Thus the trace measure
of the positive current $S\wedge d\chi_{R}\wedge J(d\chi_{R})$ tends
to 0 with $R$. Thus we find that the second term, $II$, tends to
0 as $R\rightarrow\infty.$ In conclusion, we see that \[
\lim_{R\rightarrow\infty}\left\langle d\chi_{R}\wedge S,\pi^{*}(v\wedge w)\right\rangle =0,\]
for every pair of forms $v$ and $w$ as above. Since every $(n-p-1,n-p)-$form
$\alpha$ can be written as a linear combination of forms of the type
$v\wedge w$ as above, this implies that \[
\lim_{R\rightarrow\infty}\left\langle d\chi_{R}\wedge S,\pi^{*}(\alpha)\right\rangle =0.\]
By definition, this mean precisely that $d(\pi_{*}(dd^{\#}f_{1}\wedge...\wedge dd^{\#}f_{p}))=0,$
as desired. 
\end{proof}

\section{Tropical geometry\label{sec:Tropical-geometry}}

For a finite set $A$ in $\mathbb{Z}^{n}$ we let $P=conv(A),$ the
convex hull in $\mathbb{R}^{n}$ of the set $A$. 
\begin{defn}
\label{def:tropicalpolynomial}A tropical polynomial is a function
$f(x)=\max_{\alpha\in A}(-\nu(\alpha)+\alpha\cdot x)$, where $\nu:A\rightarrow\mathbb{R}$
is some arbitrary function. For a tropical polynomial $f$, we define
the associated tropical hypersurface, which we will denote $V_{f}$,
as the set where $f$ is not smooth.
\end{defn}
Observe that a tropical polynomial is a convex function. Moreover,
since the maximum of a finite number of affine functions is piecewise
affine, we see that $V_{f}$ is the set where $f$ is not affine.
This set coincides with the set where two or more of the elements
which we take the maximum over obtain the maximum value at the same
time. It is easy to realize that $V_{f}$ thus consist of finitely
many affine hyperplanes (or rather convex polyhedras), glued together
at $(n-2)-$dimensional affine manifolds of $\mathbb{R}^{n}$. Now,
let us extend the function $\nu$ to all of $\mathbb{R}^{n}$ by letting
\begin{eqnarray*}
\nu_{\infty}(x) & = & \begin{cases}
\nu(x), & x\in A\\
\infty, & x\notin A\end{cases}.\end{eqnarray*}
The tropical polynomial $f$ then coincides with the Legendre transform
of $\nu_{\infty}$. It is a well known fact of convex analysis that
applying the Legendre transform twice to any function $g:\mathbb{R}^{n}\rightarrow\mathbb{R}$
will produce the largest convex function which is smaller than $g$
at any point. Thus, applying the Legendre transform to the tropical
polynomial $f$ gives us the largest convex function on $\mathbb{R}^{n}$
which, when restricted to $A$, is less than or equal to $\nu$. We
will denote this function by $\tilde{\nu}$. It is not hard to realize
that $\tilde{\nu}$ is piecewise affine on $P$ and equal to $+\infty$
on $\mathbb{R}^{n}\setminus P$. Let us assume for the time being
that the dimension $n=2.$ Then the set $\Gamma\subset P$ defined
as the set where the function $\tilde{\nu}$ is singular, is a graph
which is dual to the tropical line $V_{f}$ in the following sense
(cf. \cite{RST}): each edge of $\Gamma$ is perpendicular to an edge
of $V_{f}$ and vice versa. One calls the graph $\Gamma$ a \emph{convex
triangulation} of the polytope $P.$ Similar statements hold in higher
dimensions as well. 

We can associate weights, normal vectors and primitive integer vectors
to the facets of $V_{f}$ in the following way: consider an $(n-1)-$dimensional
facet $V$ of $V_{f}$. The set $V$ is the set where precisely two
of the affine functions competing for the maximum in the tropical
polynomial $f$, say $-\nu(\alpha_{1})+\alpha_{1}\cdot x$ and $-\nu(\alpha_{2})+\alpha_{2}\cdot x$,
are equal and realize the maximum. The facet $V$ has two natural
normal vectors defined from the data given, namely $\alpha_{1}-\alpha_{2}$
and $\alpha_{2}-\alpha_{1}$. We pick one of these two, and call it
the normal vector $v$ associated to $V$. Note that a choice of normal
vector $v$ induces an orientation on $V$ compatible with any fixed
choice of orientation on $\mathbb{R}^{n}$. We will only be interested
in the pair $(V,v)$ where we take the orientation of $V$ into account,
and consequently it does not matter which vector we chose above to
be the normal vector associated to $V$: If we instead had chosen
$-v$, the orientation of $V$ would have been reversed. The weight
$w$ is the absolute value of the greatest common divisor of the numbers
$v^{1},...,v^{n}$, where $v^{j}$ denotes the $j$:th component of
the vector $v.$ We let $N$ denote the primitive integer vector associated
to $v=\alpha_{1}-\alpha_{2}$, i.e., the vector in $\mathbb{Z}^{n}$
such that $wN=v$.

A tropical hypersurface $V_{f}$ thus consists of a finite number
of convex polyhedras of dimension $(n-1)$, call them $V_{1},...,V_{s}$,
which are glued together along convex polyhedras of dimension $(n-2)$,
which we denote by $W_{1},...,W_{r}$. Now, assume that $W_{1}$ is
the locus of intersection of $V_{1},...,V_{k}$, and the sign of the
corresponding normal vector $v_{i}$ has been chosen such that each
$V_{i}$ induces the same orientation on $W_{1}$. We explain this
last condition in detail. Fix an orientation of $W_{1}$ and let $\{e_{1},..,e_{n-2}\}$
be a basis of $W_{1}$, compatible with the orientation chosen. Fix
also one of the convex polyhedras adjacent to $W_{1}$, say $V_{1}$.
Then there exists a unique unit normal of $W_{1}$ pointing into $V_{1}$,
which we denote by $w_{1}$, and for which $\{w_{1},e_{1},..,e_{n-2}\}$
is a basis for $V_{1}$. We choose the sign of $v_{1}$ so that $\{v_{1},w_{1},e_{1},..,e_{n-2}\}$
is a basis compatible with the fixed orientation of $\mathbb{R}^{n}$.
Under these circumstances one can show:
\begin{prop}
\label{pro:Balancing-property}(Balancing property of tropical varieties,
cf. \cite{Mikhalkin}) With the above hypothesis the\emph{ balancing
condition} holds around $W$:\[
\sum_{i=1}^{k}v_{i}=0.\]

\end{prop}
Let us study an example as to see how tropical polynomials may arise
in practice. 

Consider a complex algebraic hypersurface $\{h=0\}$ in $\mathbb{C}^{n}$,
where $h$ is a Laurent polynomial $h=\sum_{\alpha\in A}c_{\alpha}z^{\alpha}$,
where multi-index notation is used. Let $P=conv(A)$ be the \emph{Newton
polytope} associated with $f$. We consider the function $Log:\mathbb{C}^{n}\rightarrow\mathbb{R}^{n},$
given by $Log(z_{1},...,z_{n})=(\log|z_{1}|,...,\log|z_{n}|),$ and
define $\mathcal{A}_{h}=Log(\{h=0\}).$ This set $\mathcal{A}_{h}\subset\mathbb{R}^{n}$
is called the \emph{amoeba }of the polynomial $h$. Tropical pictures
arise when we start deforming the amoeba, and shrink its width to
0. To make this precise, we let for $t>0$, $\log_{t}(x)=\log(x)/\log(t),$
and define $Log_{t}$ by exchanging $\log$ for $\log_{t}$ in the
definition of $Log$. Also we define $\mathcal{A}_{h}^{t}=Log_{t}(\{h=0\}).$
Then as $t\rightarrow0$, the sets $\mathcal{A}_{h}^{t}$ converges
to a set in the Hausdorff topology (cf. \cite{Mikhalkin}), which
we denote by $\mathcal{S}_{h}$. This set $\mathcal{S}_{h}$ can actually
be seen to be piecewise affine, and all its pieces have rational slope.
Thus, it constitutes a tropical variety. 
\begin{example}
Let us consider the two dimensional case. We choose the polynomial
$h$ to be $1+z^{2}+w$ where $(z,w)$ are coordinates for $\mathbb{C}^{2}$.
Its Newton polytope is then the triangle with vertices at the points
$(0,0),(2,0),$ and $(0,1)$. When considering the image of $\mathcal{A}_{h}\subset\mathbb{R}^{2},$
where we denote the coordinates on $\mathbb{R}^{2}$ with $(x,y),$
under the map $Log$, points at which one of the coordinates is 0
will be sent to $-\infty,$ so we start searching for those. If $w=0$
then $z=\pm1$, if we are to have $(z,w)\in\{h=0\}$. This point will
be sent under the $Log$-map to the ray along the $y-$axis starting
at $(0,0)$ and ending in $(0,-\infty).$ Similarly, if $z=0$ then
$w=-1,$ so this point will be sent to the ray along the $x-$axis,
starting at $(0,0)$ ending in $(-\infty,0).$ Also, for $z$ and
$w$ large, and $(z,w)\in\{h=0\}$, we have that $\log|w|\thickapprox2\log|z|,$
that is $y\thickapprox2x$. Thus the amoeba will have three asymptotic
lines, namely the sets $(-\infty,0]\times\{0\}$, $\{0\}\times(-\infty,0]$
and $\{y=2x\}$. Moreover, one can show that each component of the
amoeba is convex (cf. \cite{Passare}). It is not hard to realize
that if we consider the limit of $\mathcal{A}_{f}^{t}$ as $t$ gets
closer and closer to 0, the picture is that the {}``deformed'' amoeba
converges to exactly the asymptotic lines we have found, and we obtain
the tropical curve given by the tropical polynomial $\max\{0,y,2x\}$.
At this point we should also note that each of the directional vectors
for the lines are in fact normal vectors to the Newton polytope. As
in the above discussion, we say the the $\mathcal{S}_{h}$ is dual
to the polytope $P$. 
\end{example}
Tropical geometry can also be seen as algebraic geometry over a non-archimedian
field, $\mathbb{K}$. The attribute non-archimedian means that the
field has a norm which satisfy a stronger condition than the triangle
inequality, namely that \[
|x+y|\leq\max\{|x|,|y|\}.\]
Here we let $\mathbb{K}$ be the field of Puiseux series, namely the
set of all formal power series $\sum_{q\in\mathbb{Q}}a_{q}t^{q},$
where we demand that the set of all $q$ such that $a_{q}\neq0$ is
bounded from below. We can equip $\mathbb{K}$ with a valuation map
$\nu:\mathbb{K}\rightarrow\mathbb{R}$, by demanding that $\nu(\sum_{q\in\mathbb{Q}}a_{q}t^{q}\}$
is the infimum of all $q$ such that $a_{q}\neq0$. For instance $\nu(3t^{-22}+2t^{2}+t^{4}+4)=-22$.
Let us now consider the polynomial ring $\mathbb{K}[z_{1},...,z_{n}]$,
and an element in it, $G$. Thus $G=\sum_{\alpha_{i}\in A}r_{\alpha_{i}}z^{\alpha_{i}}$,
for $A$ some finite subset of $\mathbb{Z}^{n}$, and $r_{\alpha_{i}}\in\mathbb{K}$.
To this $G$ we associate its \emph{tropicalization \[
trop(G)(x)=\max_{\alpha_{i}\in A}{\{\alpha_{i}\cdot x-\nu(r_{\alpha_{i}})\}},\]
}where $x\in\mathbb{R}^{n}$, and $\alpha\cdot x$ is the scalar product
between $x$ and $\alpha$. For instance, $trop(t^{-2}zw+tw^{2})(x,y)=\max{\{-2+x+y,1+2y\}}$.
Similar as for the function $Log$ defined above, we put $Val(z_{1},...,z_{n})=(val(z_{1}),...,val(z_{n}))$.
We now come to an important point: One can show (\cite{Kapranov})
that the closure of the set $Val(\{G=0\})$ is equal to the set in
$\mathbb{R}^{n}$ where the maximum $tropG(x)=\max_{\alpha_{i}\in A}{\{\alpha_{i}\cdot x-\nu(r_{\alpha_{i}})\}}$
is obtained by two or more of the $\alpha_{i}$. Thus the closure
of the set $Val(\{G=0\})$ defines a tropical variety. Of course,
by letting the function $\nu$ in definition \ref{def:tropicalpolynomial}
be equal to $\alpha_{i}\mapsto\nu(r_{\alpha_{i}})$, the tropical
polynomial will be just $trop(G)$, and so we could equally well take
the following as a definition of a tropical variety:
\begin{prop}
(\cite{Kapranov}) A tropical variety is given by the closure of $Val(\{G=0\})$
in $\mathbb{R}^{n}$, where $G\in\mathbb{K}[z_{1},...,z_{n}].$\end{prop}
\begin{rem}
Let $h=\sum_{\alpha\in A}c_{\alpha}z^{\alpha}$ be a complex polynomial
with $h(0)=0$, where $A$ is a finite subset of $\mathbb{Z}^{n}$.
In the complex setting, there is a generalization of the Lelong number,
called Kisleman's directed Lelong number, denoted $\gamma_{z,\varphi}(x)$,
which gives more precise information concerning the singularities
of $\varphi.$ It depends on three parameters: a plurisubharmonic
function $\varphi$ on $\mathbb{C}^{n}$, a point $z\in\mathbb{C}^{n}$,
and a vector in $x\in\mathbb{R}_{+}^{n}$, and if $x=(1,...,1)$ it
reduces to the ordinary (complex) Lelong number. It is well known
that $x\mapsto\gamma_{z,\varphi}(x)$ is a convex function, for $x\in\mathbb{R}_{+}^{n}$.
For our discussion, it suffices to say that, with $\varphi(z)=\log(h(z)),$
\[
\gamma_{0,\varphi}(x)=\max_{\alpha\in A}\{\alpha\cdot x\}.\]
This corresponds to choosing $\mathbb{K}=\mathbb{C}$ and endowing
$\mathbb{K}$ with the trivial valuation $\nu(w)=0$ for every $w\in\mathbb{C}$.
Indeed, we then have \[
trop(h)(x)=\max_{\alpha\in A}\{\alpha\cdot x\}=\gamma_{0,\log(h(z))}(x).\]
For $\nu$ a valuation on $\mathbb{K}$, we put \[
trop_{v}(\log(h))(x)=\max_{\alpha_{i}\in A}{\{\alpha_{i}\cdot x-\nu(r_{\alpha_{i}})\}}.\]
Thus, if we restrict ourselves to complex polynomials, $trop_{\nu}$
could be considered as a generalization of Kiselman's Lelong number.
It would be interesting to know if one could extend $trop_{v}$ to
act on arbitrary plurisubharmonic functions. 
\end{rem}

\subsection{Tropical geometry and super forms}

Since a tropical polynomial $f$ is the maximum of a finite number
of $\mathbb{Z}-$affine functions on $\mathbb{V}$ and since $dd^{\#}f=0$
at points where $f$ is affine, we must have that $Supp(dd^{\#}f)=V_{f}$.
Thus the support of the current $dd^{\#}f$ coincides as a point set
with the tropical hypersurface $V_{f}$. We make the following definition. 
\begin{defn}
The support of a current $T$ is of dimension $(n-1)$ if $Supp(T)$
is a piecewise smooth manifold of dimension $(n-1)$, that is, $Supp(T)$
consists of a finite number of smooth manifolds of dimension $n-1$,
glued together along manifolds of lower dimension. 
\end{defn}
We can now prove the fundamental result of this paper.
\begin{prop}
\label{pro:fundamentaltheorem}There is a one to one correspondence
between tropical hypersurfaces $V_{f}$, and closed, positive (1,1)-currents
$T$ whose support is of dimension $n-1$, and whose \emph{normal
vectors} (see below) are integral.\end{prop}
\begin{proof}
Let $T$ be as in the hypothesis, and denote by $A$ the support of
$T$. We first show that $A$ is a piecewise affine manifold. Fix
a point $x\in A$, and a small ball $B$ centered at $x,$ such that
$B\setminus A$ consists of precisely two components; call them $C_{1}$
and $C_{2}$. By Proposition \ref{pro:suppconvex} both $C_{1}$ and
$C_{2}$ are convex, which implies that each ($n-1$)-dimensional
piece of $A$ must be affine. Thus $A$ is piecewise affine. Now,
by Proposition \ref{pro:ddsharplemma}, we can find a convex function
$f$ such that $T=dd^{\#}f$. Let us denote by $V_{1},...,V_{N}$
the $(n-1)-$dimensional (affine) pieces of $A$. For each $i=1,...,N$,
there is a vector $v_{i}\in\mathbb{V}$ and a real number $c_{i}$
such that, for $x\in V_{i}$, \[
f(x)=-c_{i}+v_{i}\cdot x.\]
The convexity of $f$ implies that, in fact, \[
f(x)=\max_{i=1,...,N}(-c_{i}+v_{i}\cdot x),\]
for $x\in\mathbb{V}$. The condition that the currents \emph{normal
vectors} are integral, means that each vector $v_{i}$ belongs to
$\mathbb{Z}^{n}$ (under the identification of $\mathbb{V}$ with
$\mathbb{R}^{n}$). If this is the case, then $f$ is actually a tropical
polynomial, and thus we can conclude that $Supp(T)$ coincides with
a tropical hypersurface. This establishes one part of the correspondence.
The second part is easier: for each tropical hypersurface $V_{f}$
we let $T=dd^{\#}f$; then $T$ satisfies the hypothesis of the proposition.
\end{proof}
The arguments used above immediately give:
\begin{cor}
\label{cor:supportcodim1impliesaffine}If the support $Supp(T)$ is
of dimension $n-1,$ for a closed, positive $(1,1)-$current $T$,
then $Supp(T)$ is piecewise affine.
\end{cor}
By the previous discussion, we know that $Supp(T)=V_{f}$ consists
of a finite number of convex polyhedras $V_{i}$, whose affine hull
is of dimension $n-1$, glued together at affine convex polyhedras
$W_{k}$ of dimension $n-2$. Let us recall the discussion before
Proposition \ref{pro:Balancing-property}: each facet $V_{i}$ is
the set where, for some $\alpha_{1},\alpha_{2}\in A$, \[
-c_{\alpha_{1}}+v_{\alpha_{1}}\cdot x=-c_{\alpha_{2}}+v_{\alpha_{2}}\cdot x\]
 attains the maximum defining the tropical polynomial $f,$ and we
defined the normal vector of $V_{i}$ to be, up to sign, equal to
$v_{i}=\alpha_{1}-\alpha_{2}$ . Thus $v_{i}$ is a normal vector
to $V_{i}$ whose length is determined by the tropical polynomial.
We make the following definition:
\begin{defn}
The normal 1-form associated to $V_{i}$ is defined as $v_{i}^{*}=d(v_{i}\cdot x)$
. 
\end{defn}
Let $V$ be a hyperplane in $\mathbb{R}^{n}$ with normal $v$, and
let $\delta_{V}$ denote the surface measure of $V.$ We will consider
the $(1,1)-$current \[
\frac{{1}}{|v|}\delta_{V}v^{*}\wedge J(v^{*}),\]
 whose action on an $(n-1,n-1)$-form $\alpha$ is defined by\[
\left\langle \frac{{1}}{|v|}\delta_{V}v^{*}\wedge J(v^{*}),\alpha\right\rangle =\frac{{1}}{|v|}\int_{V}\frac{v^{*}\wedge J(v^{*})\wedge\alpha}{\omega_{n}}\delta_{V}.\]
Observe that this current does not depend on which sign we have chosen
for the normal vector $v$. This current represent a tropical variety:
\begin{prop}
\label{lem:ddsharpformula}Let $V\subset\mathbb{R}^{n}$ be a hyperplane,
determined by a normal vector $v=(v_{1},...,v_{n}),$ and define $f=\max{(0,v\cdot x)}$.
Then \[
dd^{\#}f=\frac{{1}}{|v|}\delta_{V}v^{*}\wedge J(v^{*}).\]
Moreover, we have the following equality \begin{equation}
[V]\wedge J(v^{*})=\frac{{1}}{|v|}\delta_{V}v^{*}\wedge J(v^{*}),\label{eq:ddsharpformulaequation}\end{equation}
where $[V]$ is the current of integration on $V,$ with orientation
determined by $v.$ \end{prop}
\begin{proof}
We prove equation \eqref{eq:ddsharpformulaequation} first. We have
the following equality of currents on $\mathbb{R}^{n}$: \begin{equation}
\frac{{1}}{|v|}\delta_{V}v^{*}=[V],\label{eq:tempformula}\end{equation}
where $[V]$ is the current of integration of $V,$ defined in a natural
way as \[
\left\langle [V],\alpha\right\rangle =\int_{V}\alpha,\]
 for $\alpha$ a compactly supported $dimV$-form on $\mathbb{R}^{n}.$
To prove this we extend $\{v\}$ to an orthonormal basis of $\mathbb{R}^{n}$,
compatible with the orientation chosen, which we denote $\{|v|^{-1}v,e_{1},...,e_{n-1}\}$.
For simplicity, we use the notation $e^{*}=e_{1}^{*}\wedge...\wedge e_{n-1}^{*}$.
Then we need only to prove the formula for forms of the type $\alpha_{0}e^{*},$
where $\alpha_{0}$ is a function. But, since $\delta_{V}=e^{*}$,
and $v^{*}\wedge e^{*}=|v|dx_{1}\wedge...\wedge dx_{n}$, we see that
\[
\left\langle \frac{\delta_{V}}{|v|}v^{*},\alpha_{0}e^{*}\right\rangle =\left\langle \delta_{V},\alpha_{0}|v|^{-1}v^{*}\wedge e^{*}\right\rangle =\int_{V}\alpha_{0}e^{*}.\]
On the other hand, \[
\left\langle [V],\alpha_{0}e^{*}\right\rangle =\int_{V}\alpha_{0}e^{*},\]
which proves the formula \eqref{eq:tempformula}. Thus, we see that
\[
\frac{{1}}{|v|}\delta_{V}v^{*}\wedge J(v^{*})=[V]\wedge J(v^{*}).\]

We now proceed to prove the first formula of the proposition. Recall
that if $P'\subset V$ is a submanifold of the same dimension as that
of $V$, and with piecewise smooth boundary, then the current $T:=[P']\wedge J(v^{*})$
satisfies\[
\left\langle dT,\alpha\right\rangle =\left\langle [\partial P']\wedge J(v^{*}),\alpha\right\rangle ,\]
which follows from Stokes' theorem \eqref{pro:stokes}. 

We begin by considering the function $f=\max{\left\{ 0,x_{n}\right\} }$.
To compute $dd^{\#}f$ we choose for $\epsilon>0$, a family of smooth,
one-variable functions $g_{\epsilon}$ satisfying $\lim_{\epsilon\rightarrow0}g_{\epsilon}(t)=\max{\{0,t\}}$,
and $\lim_{\varepsilon\rightarrow0}g_{\varepsilon}^{''}=\delta_{0}$
(the Dirac measure at 0). For such a family $g_{\epsilon}$, we put
\[
f_{\varepsilon}(x_{1},..,x_{n})=g_{\varepsilon}(x_{n}).\]
Then, for each $\alpha\in\mathcal{E}^{n-1,n-1}$, \[
\left\langle dd^{\#}f_{\varepsilon},\alpha\right\rangle =\left\langle f_{\varepsilon}^{''}dx_{n}\wedge d\xi_{n},\alpha\right\rangle =\left\langle f_{\varepsilon}^{''}dx_{n}\wedge d\xi_{n},\alpha_{nn}\widehat{dx_{n}\wedge d\xi_{n}}\right\rangle \]
\[
=\int_{\mathbb{R}^{n}}f_{\varepsilon}^{''}\alpha_{nn}dx_{1}\wedge...\wedge dx_{n},\]
 where $\alpha_{nn}$ is the coefficient in front of $\widehat{dx_{n}\wedge d\xi_{n}}$
in the sum defining $\alpha$. Thus we see that \begin{equation}
\lim_{\varepsilon\rightarrow0}\left\langle dd^{\#}f_{\varepsilon},\alpha\right\rangle =\int_{\mathbb{R}^{n-1}}\alpha_{nn}(x_{1},...,x_{n-1},0)dx_{1}\wedge...\wedge dx_{n-1},\label{eq:ddsharpofsinglevariablemax}\end{equation}
 which is the same as saying \[
dd^{\#}f=dd^{\#}\max{\left\{ 0,x_{n}\right\} }=[\{x_{n}=0\}]\wedge d\xi_{n}.\]
Now, let's turn to the general case: as above, let $\{v/|v|,e_{1},...,e_{n-1}\}$
be an orthonormal basis, and let $F$ correspond to the matrix $(v/|v|,e_{1},....,e_{n-1})$
in the standard basis of $\mathbb{R}^{n}$. We again use the notation
$e^{*}=e_{1}^{*}\wedge...\wedge e_{n-1}^{*}$. Then, if we consider
the action of $dd^{\#}f$ on the form $\alpha e^{*}\wedge J(e^{*}),$
since $detF=1$ we get by \eqref{eq:pullbackpushforwardchangeofvariable}
and the discussion above, that \[
\left\langle dd^{\#}max(0,v\cdot x),\alpha e^{*}\wedge J(e^{*})\right\rangle =\]
\[
=\left\langle dd^{\#}max(0,|v|x_{n}),F^{*}(\alpha e^{*}\wedge J(e^{*}))\right\rangle =\]
\[
=|v|\int_{\{x_{n}=0\}}F^{*}(\alpha e^{*})=|v|\int_{\{v\cdot x=0\}}\alpha e^{*}.\]
 On the other hand,\[
\left\langle [\{v\cdot x=0\}]\wedge J(v^{*}),\alpha e^{*}\wedge J(e^{*})\right\rangle =|v|\int_{\{v\cdot x=0\}}\alpha e^{*},\]
 where we used that $\int_{\mathbb{W}}J(v^{*})\wedge J(e^{*})=|v|$.
Since we need only to consider forms that are multiples of $e^{*}\wedge J(e^{*})$,
we have proved that \[
[\{v\cdot x=0\}]\wedge J(v^{*})=dd^{\#}max(0,v\cdot x),\]
which is what we aimed for. 
\end{proof}
For a tropical hypersurface $V_{f}$ consisting of $(n-1)$-dimensional
convex polyhedras $V_{i}$ as discussed before, we consider the current
defined as $T^{'}=\sum_{i=1}^{N}[V_{i}]\wedge J(v_{i}^{*})$, and
let $T=dd^{\#}f.$ The previous proposition shows that $Supp(T-T^{'})$
is of dimension at most $n-2$. Moreover, $T-T^{'}$ is closed. These
hypotheses actually implies that $T-T^{'}=0$ as follows from the
following lemma. 
\begin{lem}
\label{lem:thmofsupport}Let $S$ be a closed $(p,q)-$current whose
coefficients are measures and whose support is a piecewise affine
manifold $M\subset\mathbb{R}^{n}$ of co-dimension $p+1$. Then $S=0.$ \end{lem}
\begin{proof}
Let $x\in M$. Assume that, for a small neighbourhood $U$ of $x$,
we can choose coordinates so that $M\cap U=\{x_{1}=x_{2}=...=x_{p+1}=0\}$.
Since $S$ has measure coefficients, it is easy to see that \[
x_{1}S=x_{2}S=...=x_{p}S=0\]
on $U$. It follows that $d(x_{1}S)=dx_{1}\wedge S=0$, thanks to
$S$ being closed. Thus, we can write $S=S^{'}\wedge dx_{1}$ for
a $d-$closed $(p-1,q)-$current $S^{'}$. By the same means $x_{2}S^{'}=0$
from which we see that $dx_{2}\wedge S^{'}=0$ and so $S^{'}=S^{''}\wedge dx_{2}$
for some $(p-2,q)-$current $S^{''}.$ Repeating the argument, we
eventually find that there is a $(0,q)-$current $S^{'''}$ such that
$S=S^{'''}\wedge dx_{1}\wedge...\wedge dx_{p}$. As before, this $S^{'''}$
satisfies the equation $S^{'''}x_{p+1}=0,$ which implies $S^{'''}\wedge dx_{p+1}+dS^{'''}\cdot x_{p+1}=0$.
Thus $S^{'''}\wedge dx_{p+1}=0$ on $M$, and since $S^{'''}=\sum_{|J|=q}S_{J}d\xi_{J}$
for some measures $S_{J,}$ we see that $S^{'''}$, and hence $S$
, vanish on $U$. Thus we have shown that $S$ carries no mass on
the pieces of $Supp(S)$ which are of pure co-dimension $p+1$. Thus
$S$ has support that is a piecewise affine manifold of co-dimension
$p+2.$ Iterating the procedure above gives the desired result. 
\end{proof}
Concluding the discussion before the lemma, we obtain the following
result:
\begin{prop}
\label{pro:ddsharpascurrentsofintegration}With the notation above,
$dd^{\#}f=\sum[V_{i}]\wedge J(v_{i}^{*})$. 
\end{prop}
We can also show the following proposition, sheding more light on
the connection between currents and tropical hypersurfaces. 
\begin{prop}
Let $T=\sum[V_{i}]\wedge J(v_{i}^{*})$ be a tropical hypersurface.
Then the condition $dT=0$ is equivalent to the balancing condition
(cf. Proposition \ref{pro:Balancing-property}): for each $(n-2)-$dimensional
affine manifold $W$ defined as the locus where $m$ hyperplanes $V_{1},...,V_{m}$
meet, we have that that $\sum_{1}^{m}v_{i}=0$. \end{prop}
\begin{proof}
For $1\leq j\leq n$ , we let $v_{i}^{j}$ denote the $j$:th component
of the vector $v_{i}.$ Each $V_{i}$ has a boundary built up from
a number of $(n-2)$ dimensional pieces which we call $P_{i}^{r}$.
For a fixed $W$ where $V_{1},...,V_{m}$ meet, we then have finitely
many $P_{i}^{r}$, say $P_{1}^{1},...,P_{k}^{1}$ coinciding with
$W$ as sets. Fix a $(n-2,n-1)$-form $\alpha=\sum_{I,j}\alpha_{I,\hat{j}}dx_{I}\wedge d\widehat{\xi_{j}}$
with compact support in a small neighbourhood of a point on $W$.
We can choose the support so small that $W$ is the only part of co-dimension
two of $V_{f}$ that lies in $Supp(\alpha)$. Then, by Stokes' theorem,
\[
\left\langle dT,\alpha\right\rangle =\left\langle \sum_{i=1}^{m}[\partial V_{i}]\wedge J(v_{i}^{*}),\alpha\right\rangle =\left\langle \sum_{i=1}^{m}[P_{i}^{1}]\wedge J(v_{i}^{*}),\alpha\right\rangle .\]
 Thus, since $J(v_{i}^{*})\wedge\widehat{d\xi_{j}}=v_{i}^{j}d\xi_{1}\wedge...\wedge d\xi_{n}$
, \[
\left\langle dT,\alpha\right\rangle =\sum_{i=1}^{m}\left\langle [P_{i}^{1}]\wedge J(v_{i}^{*}),\sum_{I,j}\alpha_{I,\hat{j}}dx_{I}\wedge\widehat{d\xi_{j}}\right\rangle =\pm\sum_{j,I}(\sum_{i=1}^{n}v_{i}^{j})\int_{W}\alpha_{I,\hat{j}}dx_{I}.\]
 Thus, if $\sum_{i=1}^{m}v_{i}=0,$ this last sum is 0, whence $\left\langle dT,\alpha\right\rangle =0.$
Since we can do this for every $W$, we see that $dim(Supp(dT))\leq n-3$
and since $dT$ has measure coefficients, Lemma \eqref{lem:thmofsupport}
implies that $dT=0$. Conversely, if $\left\langle dT,\alpha\right\rangle =0,$
then \[
\sum_{j,I}(\sum_{i=1}^{n}v_{i}^{j})\int_{W}\alpha_{I,\hat{j}}dx_{I}=0\]
and so, to see that the balancing property holds, it suffices to choose,
for every fixed $I_{0},j_{0}$, a form $\alpha$ such that $\int_{W}\alpha_{I_{0},\hat{j_{0}}}dx_{I_{0}}=1$
and $\int_{W}\alpha_{I,\hat{j}}=0$ for $I\neq I_{0},J\neq J_{0}$. \end{proof}
\begin{example}
For the tropical hyperplane corresponding to the polynomial $f=\max{\{0,v\cdot x\}}$
we have\[
\nu_{0}(V_{f})=\nu_{0}(dd^{\#}f)=|v|.\]

To see this, let $V$ be the singularity locus of $f$. Then, since
$T:=dd^{\#}f=[V]\wedge J(v^{*}),$ we have \[
\Theta_{T}(B(0,r))=\frac{1}{2^{n-1}(n-1)!}\sum_{i=1}^{n}v_{i}\int_{\{B(0,r)\cap V\}\times\mathbb{R}^{n}}d\xi_{i}\wedge(dd^{\#}|x|^{2})^{n-1}\]
 and it easy to see that, \[
\frac{1}{2^{n-1}(n-1)!}\int_{\{B(0,r)\cap V\}\times\mathbb{R}^{n}}d\xi_{i}\wedge(dd^{\#}|x|^{2})^{n-1}=\mbox{\ensuremath{\frac{{v_{i}}}{|v|}}}Vol(B^{n-1})r^{n-1}.\]
Thus\[
\Theta_{T}(B(0,r))=Vol(B^{n-1})|v|r^{n-1}.\]
and so \[
\nu_{x}(dd^{\#}f)=\lim_{r\rightarrow0}\frac{{\Theta_{T}(B(x,r))}}{r^{n-1}}=|v|,\]
 as claimed.
\end{example}
An easy adaptation of this example shows the following:
\begin{prop}
Let $T=\sum[V_{i}]\wedge J(v_{i}^{*})$ and let $x$ be a point where
a finite number of the polyhedras $V_{i}$ meet at a convex polyhedron
of co-dimension 2. Assume, after reordering, that they are $V_{1},...,V_{m}.$
Then \[
\nu_{_{x}}(T)=\frac{1}{2}\sum_{i=1}^{m}|v_{j}|.\]

\end{prop}

\subsection{Intersection theory and tropical varieties of higher co-dimension}

Let $f_{1},...,f_{p}$ be tropical polynomials with corresponding
tropical hypersurfaces $V_{f_{1}},...,V_{f_{p}}.$ 
\begin{defn}
The intersection of $V_{f_{1}},...,V_{f_{p}}$ is defined as the strongly
positive $(p,p)$-current \[
V_{f_{1}}\wedge...\wedge V_{f_{p}}:=dd^{\#}f_{1}\wedge...\wedge dd^{\#}f_{p}.\]

\end{defn}
By Proposition \ref{pro:intersectionstability2} the intersection
is stable in the following sense: if $V_{\epsilon}\rightarrow V$,
then \[
V_{\epsilon}\wedge V_{1}\wedge...\wedge V_{p}\rightarrow V\wedge V_{1}\wedge...\wedge V_{p},\]
as $\epsilon\rightarrow0.$ However, it is important to realize that
this does not imply that the support of $V_{\epsilon}\wedge V_{1}\wedge...\wedge V_{p}$
tend to the support of $V\wedge V_{1}\wedge...\wedge V_{p}$ in the
Hausdorff topology (see Example \ref{exa:intersectionexample}). We
proceed by investigating properties of intersections of tropical hypersurfaces.
\begin{prop}
\label{pro:intersectionoftropicalhyperplanes}Let $v_{i}$ be linearly
independent vectors in $\mathbb{R}^{n}$, for $1\leq i\leq n$, and
define $f_{i}(x)=\max\{0,v_{i}\cdot x\}$, where $x\in\mathbb{R}^{n}.$
Then \[
dd^{\#}f_{1}\wedge...\wedge dd^{\#}f_{n}=c\delta_{0}\]
where \[
c=|det(v_{1},...,v_{n})|.\]
\end{prop}
\begin{proof}
Let $F$ be the linear map corresponding to the inverse of the matrix
$(v_{1},...,v_{n})$, so that $|det|F||=c^{-1}$. Then, $f_{i}\circ F=\max(0,x_{i})$,
and consequently \[
F^{*}dd^{\#}f_{i}=dd^{\#}max\{0,x_{i}\}.\]
 Using formula \eqref{eq:pullbackpushforwardchangeofvariable} we
see that for a compactly supported, smooth function $g$,\[
\left\langle dd^{\#}f_{1}\wedge...\wedge dd^{\#}f_{n},g\right\rangle =c\left\langle F^{*}(dd^{\#}f_{1}\wedge...\wedge dd^{\#}f_{n}),F^{*}g\right\rangle =\]
\[
=c\left\langle dd^{\#}max\{0,x_{1}\}\wedge...\wedge dd^{\#}max\{0,x_{n}\},g\circ F\right\rangle .\]
By Fubini's theorem and equation \eqref{eq:ddsharpofsinglevariablemax},
the last line of the above equation is equal to $c\cdot(g\circ F)(0)=c\cdot g(0)$,
which finishes the proof. 
\end{proof}
If $V_{i}$ is a tropical hyperplane of the form \[
V_{i}=dd^{\text{\#}}f_{i}=[\{v_{i}\cdot x=0\}]\wedge J(v_{i}^{*})\]
 for each $1\leq i\leq q$, the same argument as in the above proof
gives us that \[
V_{1}\wedge...\wedge V_{q}=dd^{\#}f_{1}\wedge...\wedge dd^{\#}f_{q}=[\{v_{1}\cdot x=0\}\cap...\cap\{v_{q}\cdot x=0\}]\wedge J(v_{1}^{*})\wedge...\wedge J(v_{q}^{*}).\]
Notice that if two of the $q$ hyperplanes are parallel, then the
intersection vanishes. Thus, in this case the support, $Supp(V_{1}\cdot...\cdot V_{q})$,
is either of dimension $n-q$ or empty. This property holds for general
tropical hypersurfaces as well:
\begin{prop}
\label{pro:dimensiondecreasesunderintersections}Let $V_{f_{1}},...,V_{f_{p}}$
be tropical hypersurfaces such that $V_{f_{1}}\wedge...\wedge V_{f_{p}}\neq0$.
Then,\[
dim(Supp(V_{f_{1}}\wedge...\wedge V_{f_{p}}))=n-p.\]
\end{prop}
\begin{proof}
By Proposition \ref{pro:ddsharpascurrentsofintegration}, each tropical
hypersurface $V$ can be written as \[
V=\sum_{i=1}^{k}[V_{i}]\wedge J(v_{i}^{*}),\]
where each $V_{i}$ is a convex polyehedron of dimension $n-1$. Each
$[V_{i}]$ can be considered as $\chi_{i}[\tilde{V}_{i}]$ where $\tilde{V}_{i}$
denotes the affine hull of $V_{i}$, and $\chi_{i}$ is the characteristic
function \[
\chi_{i}(x)=\begin{cases}
1, & x\in V_{i}\\
0, & x\notin V_{i}\end{cases}.\]
Define the $(1,1)-$current $\tilde{V}$ by \begin{equation}
\tilde{V}=\sum_{i=1}^{k}[\tilde{V_{i}}]\wedge J(v_{i}^{*}).\label{eq:prop14formula}\end{equation}
This current is positive and closed since each summand is, and satisfies
the relation \[
Supp(V)\subset Supp(\tilde{V}).\]
Thus the inclusion \[
Supp(V_{f_{1}}\wedge...\wedge V_{f_{p}})\subset Supp(\tilde{V}_{f_{1}}\wedge...\wedge\tilde{V}_{f_{p}})\]
holds, which implies, since $dim(Supp(\tilde{V}_{f_{1}}\wedge...\wedge\tilde{V}_{f_{p}}))=n-p$,
that\[
dim(Supp(V_{f_{1}}\wedge...\wedge V_{f_{p}}))\leq n-p.\]
By the assumption that $V_{f_{1}}\wedge...\wedge V_{f_{p}}\neq0$
we see that in fact equality must hold, since $dim(Supp(V_{f_{1}}\wedge...\wedge V_{f_{p}}))<n-p$
would force $V_{f_{1}}\wedge...\wedge V_{f_{p}}=0$ by Corollary \ref{cor:supportthm}.
\end{proof}
The following result generalizes the case of positive $(1,1)-$currents:
\begin{prop}
\label{pro:pppeicewiseaffine}Let $T$ be a strongly positive $(p,p)$-current
such that $Supp(T)$ is a piecewise smooth manifold of dimension $n-p$.
Then $Supp(T)$ is piecewise affine. \end{prop}
\begin{proof}
Let $L$ be an affine subspace of $\mathbb{R}^{n}$ of dimension $n-p+1$
such that if, $\pi_{L}$ denotes the projection onto $L$, $\pi_{L}$
is proper on $Supp(T)$. By Proposition \ref{pro:dimensiondecreasesunderintersections}
this holds for almost every subspace $L$. By Proposition \ref{pro:pushforwardpropersupport},
the push-forward $(\pi_{L})_{*}T$ is a closed, positive $(1,1)-$current
on $L\times\mathbb{R}^{n}$ and by property \emph{i)}, we see that
if the push-forward is non-zero, its dimension is $n-p$. Thus, for
a generic subspace $L$ as above, $(\pi_{L})_{*}T$ is of co-dimension
$1$ in $L$, so by Corollary \ref{cor:supportcodim1impliesaffine},
we see that $Supp((\pi_{L})_{*}T)$ is piecewise affine for almost
any $\pi_{L}$. Since $Supp((\pi_{L})_{*}T)\subset\pi_{L}(Supp(T))$,
we conclude that $SuppT$ is piecewise affine.
\end{proof}
As an immediate corollary, we obtain:
\begin{cor}
With the notation of Proposition \ref{pro:dimensiondecreasesunderintersections},
\textup{$Supp(V_{f_{1}}\wedge...\wedge V_{f_{p}})$ is piecewise affine.}\end{cor}
\begin{example}
\label{exa:intersectionexample}For a simple example, assume that
we are in $\mathbb{R}^{2}$ and consider the intersection $dd^{\#}f\wedge dd^{\#}g$
where $f=\max(0,v_{1}x+v_{2}y)$, and $g=\max(0,w_{1}x+w_{2}y$) where
$(v_{1},v_{2})\neq(w_{1},w_{2})$ are vectors in $\mathbb{R}^{2}.$
This corresponds to the intersection between the lines $\{v_{1}x+v_{2}y\}$
and $\{w_{1}x+w_{2}y=0\}$. Indeed, by Proposition \ref{pro:intersectionoftropicalhyperplanes},
\[
dd^{\#}f\wedge dd^{\#}g=|det\left[\begin{array}{cc}
w_{1} & v_{1}\\
w_{2} & v_{2}\end{array}\right]|\delta_{0.}\]
 Thus the intersection of the two lines is the origin with intersection
multiplicity determined by the volume of the parallelepiped spanned
by the defining vectors of the two lines. As was noticed above, we
see that if the two lines coincide, the intersection vanishes. 

Let us take $(v_{1},v_{2})=(w_{1},w_{2})=(1,0)$ with associated tropical
lines $V$ and $W$ and let us perturb $W$ slightly by considering
the tropical line $W_{\epsilon}$ associated to $g_{\epsilon}(x)=\max\{(1+\epsilon)x_{1}+\epsilon x_{2},0\}$.
Then \[
W_{\epsilon}\cdot V=\delta_{0}|det\left[\begin{array}{cc}
1 & 1+\epsilon\\
0 & \epsilon\end{array}\right]|=\delta_{0}\cdot\epsilon,\]

and \[
W_{\epsilon}\cdot V\rightarrow W\cdot V,\]
as $\epsilon\rightarrow0.$ But, $Supp(W\cdot V)=\emptyset$ and $Supp(W_{\epsilon}\cdot V)=\{0\}$,
which shows that the condition $W_{\epsilon}\cdot V\rightarrow W\cdot V$
does not imply that $Supp(W_{\epsilon}\cdot V)\rightarrow Supp(W\cdot V).$ 
\end{example}
These results motivate the following definition: 
\begin{defn}
A tropical variety of co-dimension $k$ is the support of a strongly
positive $(k,k)$-current whose support have co-dimension $k$, where
we demand that each of the affine pieces should have rational slope. 
\end{defn}
A piece have rational slope if its affine hull, which is a plane of
co-dimension $k$, is the set where $k$ linear forms with integer
coefficients vanish.
\begin{prop}
Let $V_{1},...,V_{k}$ be tropical hypersurfaces. Then $V_{1}\wedge...\wedge V_{k}$
is a tropical variety of co-dimension $k$. \end{prop}
\begin{proof}
We know that $V_{1}\wedge...\wedge V_{k}$ is a closed, strongly positive
$(k,k)-$current. Moreover, by Proposition \ref{pro:dimensiondecreasesunderintersections}
the dimension of its support is bounded from above by $n-k$, and
since \[
Supp(V_{1}\cdot...\cdot V_{k})\subset Supp(V_{1})\cap...\cap Supp(V_{k}),\]
by equation \eqref{eq:supp_of_intersection_subset_intersection_of_supp},
each piece has rational slope. Thus it is a tropical variety of co-dimension
$k.$ 
\end{proof}
The set theoretic intersection of two tropical hypersurfaces need
not coincide with the support of a tropical variety. Indeed, if $f=\max\{{0,x,y}\}$
and $g=\max\{0,x-y\}$ then the set theoretic intersection of $V_{f}$
and $V_{g}$ is the half-ray $\{(x,y):y=x,x\geq0\}$ which is not
a tropical variety (for instance, it does not satisfy the balancing
property). However, $V_{f}\wedge V_{g}$ is equal to $\delta_{0}\omega_{n}$,
which is a tropical variety. 
\begin{rem}
The intersection theory developed here seems to fit well with the
intersection theory for tropical geometry considered in \cite{Mikhalkin2}. 
\end{rem}

\subsection{Bezout's th\label{sub:Bezout's-theorem}eorem}

We use the ideas we have developed to prove known theorems within
tropical geometry. Recall that we associated to an element $f\in\mathcal{L}$
the function \[
\tilde{f}(x)=\lim_{t\rightarrow\infty}\frac{{f(tx)}}{t},\]
and that by Proposition \ref{pro:relationbetweenfanditshomogenization}
we have the following relation between $f$ and $\tilde{f}$: \begin{equation}
\int_{\mathbb{R}^{n}\times\mathbb{R}^{n}}(dd^{\#}f)^{n}=\int_{\mathbb{R}^{n}\times\mathbb{R}^{n}}(dd^{\#}\tilde{f})^{n}.\label{eq:MAoffandtildef}\end{equation}
Let us explore the effects of this result to tropical polynomials.
For $f(x)=\max_{\alpha\in A}\{c_{\alpha}+\alpha\cdot x\}$ it is easy
to see that $\tilde{f}(x)=\max_{\alpha\in A}\{\alpha\cdot x\}.$ Relating
to the discussion in the beginning of section \ref{sec:Tropical-geometry},
this means that if $f$ corresponds to a convex triangulation of $conv(A)$
then $\tilde{f}$ corresponds to {}``forgetting'' this triangulation.
In fact, since \[
\max_{\alpha\in A}\{\alpha\cdot x\}=\max_{\alpha\in conv(A)}\{\alpha\cdot x\},\]
we see that $\tilde{f}$ is just the support function of the set $conv(A).$ 
\begin{prop}
\label{pro:mixedMAisequaltoMixedVol}Let $f_{1},...,f_{n}$ be tropical
polynomials defined by \[
f_{i}=\max_{\alpha\in A_{i}}\{c_{\alpha}^{i}+\alpha\cdot x\},\]
 where each $A_{i}$ is a finite set of point of $\mathbb{Z}^{n}$,
and $c_{\alpha}^{i}$ are real numbers. Let $\tilde{A_{i}}=conv(A)$.
Then $\tilde{f_{i}}$ is the support function of $\tilde{A_{i}}$
and \begin{equation}
\int_{\mathbb{R}^{n}\times\mathbb{R}^{n}}dd^{\#}f_{1}\wedge...\wedge dd^{\#}f_{n}=n!\cdot V(\tilde{A_{1}},...,\tilde{A_{n}}).\label{eq:bezoutsequation}\end{equation}
\end{prop}
\begin{proof}
It is clear that, if $g=f_{1}+...+f_{n}$, then $\tilde{g}=\tilde{f_{1}}+...+\tilde{f}_{n}.$
Thus, \eqref{eq:MAoffandtildef} implies that, for every $(t_{1},...,t_{n})\in\mathbb{R}^{n},$
\[
\int_{\mathbb{R}^{n}\times\mathbb{R}^{n}}(dd^{\#}(\sum_{j\in J}t_{j}f_{j}))^{n}=\int_{\mathbb{R}^{n}\times\mathbb{R}^{n}}(dd^{\#}(\sum_{j\in J}t_{j}\tilde{f}_{j}))^{n}\]
and so, by comparing coefficients, we see that\[
\int_{\mathbb{R}^{n}\times\mathbb{R}^{n}}dd^{\#}f_{1}\wedge...\wedge dd^{\#}f_{n}=\int_{\mathbb{R}^{n}\times\mathbb{R}^{n}}dd^{\#}\tilde{f_{1}}\wedge...\wedge dd^{\#}\tilde{f_{n}}.\]
Since $\tilde{f_{i}}$ is the support function of the set $\tilde{A}_{i}$,
we obtain from Proposition \ref{pro:mixedvolumeMA-1} that\[
\int_{\mathbb{R}^{n}\times\mathbb{R}^{n}}dd^{\#}f_{1}\wedge...\wedge dd^{\#}f_{n}=n!\cdot V(\tilde{A_{1}},...,\tilde{A_{n}}).\]
\end{proof}
\begin{rem}
In the complex setting, there is an analogue of the associated function
$\tilde{f}$, called the local indicator associated to a plurisubharmonic
function. See for instance the article \cite{Rashkovskii}. 
\end{rem}
Let us apply the above Proposition to obtain, in an easy way, two
results from tropical geometry. We stress that these are not new results.
Let us assume for the moment that $n=2$. A tropical curve associated
to a tropical polynomial $f=\max_{\alpha\in A}\{c_{\alpha}+\alpha\cdot x\}$
is of degree $d$ if $Newt(f)$ is equal to the set $\{(x,y):x+y\leq d,x,y\geq0\}$.
The tropical version of the Bezout theorem in 2 dimensions is the
following: 
\begin{thm}
Consider two generic tropical curves $C_{1}$ and $C_{2}$ in $\mathbb{R}^{2}$
of degree $d_{1}$ and $d_{2}$ respectively. Then the number of intersection
points counted with multiplicity is equal to $d_{1}d_{2}.$\end{thm}
\begin{proof}
If the curve $C_{i}$ corresponds to the tropical polynomial $f_{i}$
then we know that the number of intersection between the curves is
equal to \[
\int_{\mathbb{R}^{2}\times\mathbb{R}^{2}}dd^{\#}f_{1}\wedge dd^{\#}f_{2}.\]
By proposition \ref{pro:mixedMAisequaltoMixedVol} this number is
equal to $n!\cdot V(Newt(f_{1}),Newt(f_{2})).$ But, if we let $S_{d}=\{(x,y):x+y\leq d,x,y\geq0\},$
then \[
V(S_{d_{1}},S_{d_{2}})=Vol(S_{d_{1}}+S_{d_{2}})-Vol(S_{d_{1}})-Vol(S_{d_{1}})=\]
\[
=2^{-1}((d_{1}+d_{2})^{2}-d_{1}^{2}-d_{2}^{2})=d_{1}d_{2},\]
and we are done. 
\end{proof}
In general, we consider equation \eqref{eq:bezoutsequation} to be
the general version of Bezout's theorem. As another application, we
show that the Proposition implies an interesting interplay between
tropical and ordinary polynomials (cf. \cite{Sturmfels}): We begin
by recalling Bernstein's theorem.
\begin{thm}
Let $P_{1},...,P_{n}$ be (generic) polynomials on $\mathbb{C}^{n}.$
Then the number of solutions of the system $P_{1}=...=P_{n}=0$ is
equal to $n!\cdot V(Newt(P_{1}),...,Newt(P_{n})),$ where $Newt(P_{i})$
is the Newton polytope of $P_{i}.$ 
\end{thm}
For each polynomial $P_{i}$, we associate the function $\tilde{f}_{i}(x)=\sup_{\xi\in Newt(P_{i})}\xi\cdot x,$
that is, the support function of $Newt(P_{i})$. Clearly, $\tilde{f_{i}}$
belongs to $\mathcal{L}$ and is a tropical polynomial. Then Proposition
\ref{pro:mixedMAisequaltoMixedVol} combined with Bernstein's theorem
says the following: the number (counted with multiplicities) of intersection
points of the tropical hypersurfaces associated to $\tilde{f}_{i}$
is equal to the number of intersection points of the varieties $\{P_{i}=0\}\subset\mathbb{C}^{n}$. 

\bibliographystyle{amsplain}
\nocite{*}
\bibliography{supercurrents}

\end{document}